\documentclass[10pt,a4paper,reqno,oneside]{amsart}
\usepackage{amsmath,amsthm,amsfonts,amssymb}
\usepackage[usenames,dvipsnames]{color}
\usepackage{mathrsfs,mathtools}
\usepackage[extension=pdf]{hyperref}
\usepackage{todonotes}
\usepackage[top=1.65in,bottom=1.65in,left=1.5in,right=1.5in,centering]{geometry}
\definecolor{darkblue}{rgb}{0.13,0.13,0.39}
\definecolor{darkpurple}{RGB}{102,0,102}
\hypersetup{colorlinks=true,urlcolor=cyan,citecolor=darkblue,linkcolor=darkpurple,pdftitle={Continuity and strict positivity of the multi-layer extension of the SHE}}

\newtheorem{theorem}{Theorem}[section]
\newtheorem{lemma}[theorem]{Lemma}

\newtheorem{proposition}[theorem]{Proposition}
\newtheorem{corollary}[theorem]{Corollary}

\theoremstyle{remark}

\newcommand{\rd}[1] {\mathrm{d}#1}
\newcommand{\dd}[2] {\mathrm{d}#1\mathrm{d}#2}

\newcommand{\W}[2]{W(\mathrm{d}#1,\mathrm{d}#2)}
\newcommand{\mW}[3]{W^{\otimes #1}(\mathrm{d}#2,\mathrm{d}#3)}

\newcommand{\db}[1] {\mathrm{d}\mathbf{#1}}

\newcommand{\mb}[1] {\mathbf{#1}}
\newcommand{\R}{\mathbb{R}}
\newcommand{\E}{\mathbb{E}}
\newcommand{\x}{\mathbf{x}}
\newcommand{\y}{\mathbf{y}}
\newcommand{\z}{\mathbf{z}}
\newcommand{\sP}{\mathscr{P}}
\newcommand{\1}{\mathbf{1}}

\newcommand{\V}{\Vert}
\newcommand{\p}{\prime}

\usetikzlibrary{decorations.markings}

\tikzset{
  addarrows/.style={
    decorate,
    decoration={markings,
    mark=at position 2cm with {\arrowreversed[line width=1pt]{stealth}},
    mark=at position 6cm with {\arrowreversed[line width=1pt]{stealth}},
    mark=at position 10cm with {\arrowreversed[line width=1pt]{stealth}},
    mark=at position 14cm with {\arrowreversed[line width=1pt]{stealth}},
    mark=at position 18cm with {\arrowreversed[line width=1pt]{stealth}},
    mark=at position 22cm with {\arrowreversed[line width=1pt]{stealth}}
  }},
  addarrowsvert/.style={
    decorate,
    decoration={markings,
    mark=at position 2cm with {\arrow[line width=1pt]{stealth}},
    mark=at position 10cm with {\arrow[line width=1pt]{stealth}}
  }},
  cross/.style={cross out, draw, 
    minimum size=2*(#1-\pgflinewidth), 
    inner sep=0pt, outer sep=0pt}
}

\begin{document}

\title[Continuity and Strict Positivity of the multi-layer extension of the SHE]{Continuity and strict positivity of the multi-layer extension of the stochastic heat equation} 
\author[C. H. Lun]{Chin Hang Lun}
\address{C. H. Lun}
\email{chinhanglun@gmail.com}

\author[J. Warren]{Jon Warren}
\address{
  J. Warren,
  Department of Statistics,
  University of Warwick,
  Coventry,
  CV4 7AL,
  UK}
\email{j.warren@warwick.ac.uk} 
  
\maketitle
\thispagestyle{empty}

\begin{abstract}
  We prove the continuity and strict positivity of the multi-layer extension to the stochastic heat equation introduced in \cite{OW11} which form a hierarchy of partition functions for the continuum directed random polymer. 
  This shows that the corresponding free energy (logarithm of the partition function) is well defined.
  This is also a step towards proving the conjecture stated at the end of the above paper that an array of such partition functions has the Markov property.
\end{abstract}

\section{Introduction}\label{sec:intro}

In \cite{OW11} O'Connell and Warren introduced the following: for each $n = 1,2,\ldots$, $t>0$ and $x$, $y\in\R$ define
\begin{equation}
  Z_n(t,x,y) = p_t(x-y)^n \bigg(1 + \sum_{k=1}^\infty \int_{\Delta_k(t)} \int_{\R^k} R_k(\mb{s},\mb{y^\prime}; t,x,y) \;\mW{k}{\mb{s}}{\mb{y}^\prime} \bigg),
  \label{eq:ZnChaos}
\end{equation}
where $\Delta_k(t) = \{0 < s_1 < s_{2} < \cdots < s_k < t\}$, $\mb{s} = (s_1,\ldots,s_k)$, $\mb{y}^\prime = (y_1^\prime,\ldots,y_k^\prime)$ and $R_k(\mb{s}, \mb{y}^\prime ; t,x,y)$ is the $k$-point correlation function for a collection of $n$ non-intersecting Brownian bridges each of which starts at $x$ at time 0 and ends at $y$ at time $t$, see Section \ref{sec:nonIntersectingBM}.
$p_t(x-y)$ is the heat kernel $(2\pi t)^{-1/2} e^{-(x-y)^2/2t}$.
The integral is a $k$-fold stochastic integral with respect to space-time white noise, see Section \ref{sec:prelim} for the definition of such integrals.
It was shown in \cite{OW11} by considering local times of non-intersecting Brownian bridges that the infinite sum in the definition is convergent in $L^2$ with respect to the white noise.

Observe that $u = Z_1$ is the solution to the (multiplicative) stochastic heat equation (SHE) with delta initial data:
\begin{equation}
 \begin{cases}
  \partial_t u(t,x,y) = \Big( \frac{1}{2} \Delta_y + \dot{W}(t,y) \Big) u(t,x,y), \quad t\in(0,\infty), y\in\mathbb{R}, \\
  u(0,x,y) = \delta(x-y), \quad x\in\mathbb{R}.
 \end{cases}
 \label{eq:SHEDeltaX}
\end{equation}
By a solution to the above we mean a random field $u$ which satisfies almost surely the mild form 
\begin{equation}
  u(t,x,y) = p_t(x-y) + \int_0^t \int_\R p_{t-s}(y-y^\prime) u(s,x,y^\prime) \;\W{s}{y^\prime}.
  \label{eq:SHEMild}
\end{equation}
Iterating equation (\ref{eq:SHEMild}) multiple times gives the chaos expansion (\ref{eq:ZnChaos}) for $n=1$.
One can express the solution $u(t,x,y)$ in a more suggestive notation:
\begin{equation}
  u(t,x,y) = p_t(x-y) \mathbb{E}_{x,y;t}^b \bigg[ \mathscr{E}\mathrm{xp} \bigg( \int_0^t W(s,b_s) \;\rd{s} \bigg) \bigg],
  \label{eq:SHEFeymannKac}
\end{equation}
where $b$ is a Brownian bridge that starts at $x$ at time 0 and ends at $y$ at time $t$ and $\E_{x,y;t}^b$ denotes the corresponding expectation.
$\mathscr{E}\mathrm{xp}$ is the \emph{Wick exponential} defined by
\[
  \mathscr{E}\mathrm{xp}(M_t) := \exp\big( M_t - \frac{1}{2}\langle M,M\rangle_t \big),
\]
for a martingale $M$.
The Feynman--Kac formula (\ref{eq:SHEFeymannKac}) is not rigorous as it is unclear how one would define the integral of the white noise along a Brownian path and moreover to exponentiate such an expression.
However, Taylor expanding the exponential, then switching the expectation with the infinite sum and evaluating the expectation, one obtains the chaos expansion of $u$.
With this in mind, (\ref{eq:SHEFeymannKac}) can be thought of as a short hand for the chaos expansion (\ref{eq:ZnChaos}) in the case $n=1$.
On the other hand, one can obtain a rigorous expression by replacing $W$ in (\ref{eq:SHEFeymannKac}) with a smoothed version of the space-time white noise.
Indeed, Bertini and Cancrini showed in \cite{BC95} that such expression has a meaningful limit as one takes away the smoothing and that the limit solves the SHE.
With this Feynman--Kac interpretation, one can think of the solution to the stochastic heat equation as the partition function (up to a multiplication by the heat kernel) of the continuum directed random polymer \cite{AKQ12}.

Analogously, we write
\begin{equation}
  Z_n(t,x,y) = p_t(x-y)^n \mathbb{E}_{x,y;t}^X \bigg[ \mathscr{E}\mathrm{xp} \bigg( \sum_{i=1}^n \int_0^t W(s,X_s^i) \;\rd{s} \bigg) \bigg],
  \label{eq:ZnFeymannKac}
\end{equation}
where $(X_s^1,\ldots,X_s^n, 0\leq s\leq t)$ denotes the trajectories of the above mentioned collection of $n$ non-intersecting Brownian bridges and $\E_{x,y;t}^X$ is the corresponding expectation.
In the same manner as in the $n=1$ case, (\ref{eq:ZnFeymannKac}) should be thought of as the short hand for the chaos expansion (\ref{eq:ZnChaos}).
Therefore, in view of (\ref{eq:ZnFeymannKac}) one can interpret $Z_n$ as the partition function (up to a factor of the heat kernel) of a natural extension of the continuum directed random polymer involving multiple non-intersecting Brownian paths.

Since the work of Bertini and Giacomin \cite{BG97}, it is widely accepted that the logarithm of $u$ is the Cole--Hopf solution to the KPZ equation \cite{KPZ86},
\begin{equation}
  \partial_t h(t,x) = \partial_x^2 h(t,x) + \big(\partial_x h(t,x)\big)^2 + \dot{W}(t,x),
  \label{eq:KPZ}
\end{equation}
with narrow wedge initial condition.
This solution arises as the scaling limit of the corner growth model under weak asymmetry.
The Cole--Hopf solution to the KPZ equation via the Feynman--Kac formula (\ref{eq:SHEFeymannKac}) can be seen as the free energy of the continuum directed random polymer.
With this interpretation the Cole--Hopf solution can be regarded as the continuum analogue of the longest increasing subsequence of a random permutation, length of the first row of a random Young diagram, directed last passage percolation and free energy of a discrete/semi-discrete polymer in random media etc., see \cite{BDJ99}, \cite{BDJ99b}, \cite{BOO00}, \cite{Joh99}, \cite{Joh01}, \cite{PS02}, \cite{Joh03}, \cite{COSZ14} and the references therein.
In each of these discrete models, there is further structure provided either by multiple non-intersecting up-right paths on lattices, multi-layer growth dynamics or Young diagrams constructed from the RSK correspondence.
The work in the above mentioned references have shown that in some cases, utilisation of this additional structure have lead to derivations of exact formulae for the distribution of quantities of interest.
The above mentioned discrete models provide examples of what is called \emph{integrability} or \emph{exact solvability}.
The motivation for introducing the partition functions $Z_n$, which are the continuum analogue of the structures mentioned above, is that they should provide insight to the integrable structure in the continuum setting. 

The main result of this paper is that the continuum partition functions possess some nice regularity properties.

\begin{theorem}
  For all $n\geq 1$, the function $(t,x,y)\mapsto Z_n(t,x,y)$ has a version that is continuous over $(0,\infty)\times\R\times\R$. 
  Moreover,
  \[
    \mathbb{P}[Z_n(t,x,y) > 0 \text{ for all } t>0 \text{ and } x,y\in\R] = 1.
  \]
  \label{thm:ZnRegularity}
\end{theorem}

Now define for $n=1,2,\ldots$
\begin{equation}
  h_n(t,x) = \log\left( \frac{Z_n(t,0,x)}{Z_{n-1}(t,0,x)} \right),
  \label{eq:nthKPZ}
\end{equation}
with the convention that $Z_0 \equiv 1$, then $h_1(t,x)$ is the Cole--Hopf solution to the KPZ equation with narrow wedge initial data.
An immediate corollary to the above theorem is

\begin{corollary}
  For all $n\geq 1$, $h_n$ is well defined and it is a continuous function of $(t,x)$ over $(0,\infty)\times\R$.
\end{corollary}

The collection $\{h_n, n\geq 1\}$ represents a multi-layer extension to the free energy of the continuum directed random polymer.
It is the analogue in the setting of the KPZ of the multi-layer PNG or its discrete counterpart studied in \cite{PS02} and \cite{Joh03} respectively.

We mention here the work of \cite{CH13}.
The authors showed the existence of a collection of random continuous curves such that the lowest indexed curve is distributed as the time $t$ Cole--Hopf solution to the KPZ with narrow wedge initial data.
It is believed (see \cite[Conjecture 2.17]{CH13}) that for each $t>0$ fixed, their collection of curves is equal to $\{h_n(t,x): n\geq 1, x\in\R\}$ defined by (\ref{eq:nthKPZ}).
Proving this will give an alternative proof of the continuity and strict positivity of $Z_n$ at a fixed time $t$.
In this paper, we provide a direct proof of this and furthermore our proof gives a stronger result since $t$ can vary over $(0,\infty)$.

There has been other recent work on multiple polymer paths and the multilayer process in the stochastic heat equation setting. 
In \cite{DL15} and \cite{DL16}, in a manifestation of the exact solvability, the Bethe Ansatz is used to make exact and asymptotic distributional statements.
More recently in \cite{CN17}, it was shown that directed polymer models involving multiple non-intersecting random walks, each with the same starting and end points, moving through a space-time disordered environment converges to $Z_n$ defined by \eqref{eq:ZnChaos}.

The continuity and strict positivity of $u=Z_1$ was proved by considering its mild form which suggests that to prove Theorem \ref{thm:ZnRegularity} one could consider the evolution equation satisfied by $Z_n$.
By considering a smooth space-time potential, the authors in \cite{OW11} showed that $Z_n$ should satisfy a certain SPDE, see \cite[Proposition 3.3 and 3.7]{OW11}, however unfortunately it is not immediately obvious that this SPDE makes sense in the white noise setting.
Instead, we shall show that a natural extension of $Z_n$ does satisfy a rigorous evolution equation which can be regarded as a multi-dimensional stochastic heat equation.
This allows us to derive the continuity and strict positivity of the extension and from which Theorem \ref{thm:ZnRegularity} follows as a corollary.

Denote by $W_n$ the Weyl chamber $\{\mb{x}\in\R^n: x_1\geq x_2\cdots\geq x_n\}$, then for $n = 1,2,\ldots$, $t>0$ and $\mb{x}$, $\mb{y}\in W_n$ define
\begin{equation}
  K_n(t,\mb{x},\mb{y}) = p_n^*(t,\mb{x},\mb{y}) \bigg(1 + \sum_{k=1}^\infty \int_{\Delta_{k}(t)} \int_{\R^k} R_k(\mb{s},\mb{y}^\prime; t,\mb{x},\mb{y}) \;\mW{k}{\mb{s}}{\mb{y}^\prime} \bigg),
  \label{eq:KnChaos}
\end{equation}
where $R_k$ is the $k$-point correlation function of a collection of $n$ non-intersection Brownian bridges which starts at $\mb{x}$ at time 0 and ends at $\mb{y}$ at time $t$.
$p_n^*(t,\mb{x},\mb{y}) = \det[p_t(x_i-y_j)]_{i,j=1}^n$ is by the Karlin--McGregor formula \cite{KM59} the transition density of Brownian motion killed at the boundary of $W_n$.
It was proved in \cite[Proposition 3.2]{OW11} that $K_n$ also satisfies a Karlin--McGregor type formula:
\begin{equation}
  K_n(t,\x,\y) = \det[u(t,x_i,y_j)]_{i,j=1}^n,
  \label{eq:KnKarlinMcGregor}
\end{equation}
where each term in the determinant are solutions to (\ref{eq:SHEDeltaX}) each driven by the same white noise.
Now, define for $t>0$, $\x$, $\y\in W_n^\circ$
\begin{equation}
  M_n(t,\mb{x},\mb{y}) = \frac{K_n(t,\mb{x},\mb{y})}{\Delta(\mb{x})\Delta(\mb{y})},
  \label{eq:MnDefn}
\end{equation}
where $\Delta(\mb{x}) = \prod_{1\leq i<j\leq n} (x_i-x_j)$ is the Vandermonde determinant.
It follows from (\ref{eq:KnChaos}) that $M_n$ has chaos expansion
\begin{equation}
  M_n(t,\mb{x},\mb{y}) = \frac{p_n^*(t,\mb{x},\mb{y})}{\Delta(x)\Delta(y)} \bigg(1 + \sum_{k=1}^\infty \int_{\Delta_{k}(t)} \int_{\R^k} R_k(\mb{s},\mb{y}^\prime; t,\mb{x},\mb{y}) \;\mW{k}{\mb{s}}{\mb{y}^\prime} \bigg).
  \label{eq:MnChaos1}
\end{equation}
By (\ref{eq:KnKarlinMcGregor}) and the continuity of the solution to the stochastic heat equation, it is easy to see that $K_n(t,\x,\y)$ is almost surely continuous on $(0,t)\times W_n\times W_n$ and is zero on the boundary of $W_n\times W_n$.
It follows that $M_n(t,\x,\y)$ is continuous in the interior $W_n^\circ\times W_n^\circ$.
By \cite[Lemma 5.11]{BBO09}, $p_n^*(t,\x,\y)/\Delta(\x)\Delta(\y)$ is a smooth function of $(\x,\y)$ over $\R^n\times\R^n$ and since the $k$-point correlation function $R_k$ extends continuously to the boundary of the Weyl chamber, see Section \ref{sec:nonIntersectingBM}, we see from its chaos expansion (\ref{eq:MnChaos1}) that $M_n(t,\x,\y)$ is defined for $\x$, $\y\in\partial W_n$.
This also suggests that $M_n(t,\x,\y)$ is a continuous function on $W_n\times W_n$.
Furthermore, from (\ref{eq:KnKarlinMcGregor}) we see that $M_n$ being a ratio of determinants is a permutation symmetric function of its spatial variables, that is for any permutations $\pi$, $\sigma$ of $\{1,\ldots,n\}$, $M_n(t,\pi\x,\sigma\y) = M_n(t,\x,\y)$.
Hence, we can extend $M_n$ by symmetry to a function on $\R^n\times\R^n$ and we will show that there exists a version of $M_n$ that is almost surely strictly positive and continuous on the whole of $\R^n\times\R^n$ and for all $t>0$.
Moreover, when all the $\x$ coordinates are equal and likewise for $\y$, $M_n$ agrees up to a multiplicative constant with $Z_n$, that is 
\begin{equation}
  M_n(t,a\mb{1},b\mb{1}) = c_{n,t} Z_n(t,a,b),
  \label{eq:MnBoundary}
\end{equation}
where $c_{n,t} := \big(\prod_{i=1}^{n-1} i!\big)^{-1} t^{-n(n-1)/2}$ and $\mb{1} = (1,\ldots,1)$.
Equation (\ref{eq:MnBoundary}) was shown to hold in \cite{OW11} but there the continuity of $M_n$ on the boundary of $W_n$ was only established in an $L^2$ sense; here we extend it to almost sure continuity.
Note that (\ref{eq:KnKarlinMcGregor}) suggests that $K_n(t,\x,\y)$ and $M_n(t,\x,\y)$ can be regarded as the stochastic analogue of $p_n^*(t,\x,\y)$ and $p_n^*(t,\x,\y)/\Delta(\x)\Delta(\y)$ respectively where the latter has limit at the boundary equal to $c_{n,t} p_t(a-b)^n$.

In Section \ref{sec:existence}, we will show that for all $(t,\x,\y)\in(0,\infty)\times\R^n\times\R^n$, $M_n(t,\x,\y)$  satisfies almost surely the mild equation
\begin{align}
  M_n(t,\mb{x},\mb{y}) 
  &= \frac{p_n^*(t,\x,\y)}{\Delta(\x)\Delta(\y)} + \frac{1}{(n-1)!} \int_0^t \int_{\mathbb{R}^n} Q_{t-s}(\y,\y^\prime) M_n(s,\x,\y^\prime) \;\rd{\y_*^\prime} \:\W{s}{y_1^\prime} \notag \\
  &=: J_n(t,\x,\y) + I_n(t,\x,\y),
  \label{eq:MnDelta}
\end{align}
where
$\rd{\y_*^\prime} = \rd{y_2}\ldots\rd{y_n}$ and
\begin{equation}
  Q_t(\x,\y) = \frac{\Delta(\y)}{\Delta(\x)} p_n^*(t,\x,\y) = \frac{\Delta(\y)}{\Delta(\x)} \det[p_t(x_i-y_j)]_{i,j=1}^n,
  \label{eq:QtDefn}
\end{equation}
is the transition density of Dyson's Brownian motion starting from $\x\in W_n$ and ending at $\y\in W_n$.
It satisfies
\begin{equation}
  Q_t(a\mb{1},\y) = c_{n,t} \Delta(\y)^2 \prod_{i=1}^n p_t(y_i-a).
  \label{eq:QtBoundary}
\end{equation}
We can extend $Q_t$ by symmetry to a function on $\R^n\times\R^n$ and so the integral over $\R^n$ in the mild equation (\ref{eq:MnDelta}) is defined.

Consider also the following integral equation for $(t,\y) \in(0,\infty)\times\R^n$,
\begin{align}
  M_n^g(t,\y)
  &= \frac{1}{n!}\int_{\mathbb{R}^n} g(\y^\prime) Q_t(\y,\y^\prime) \;\mathrm{d}\y^\prime \notag \\
  &\qquad + \frac{1}{(n-1)!} \int_0^t \int_{\mathbb{R}^n} Q_{t-s}(\y,\y^\prime) M_n^g(s,\y^\prime) \;\rd{\y_*^\prime} \:\W{s}{y_1^\prime} \notag \\
  &=: J_n(t,\y) + I_n(t,\y),
  \label{eq:MnBounded}
\end{align}
where $g: \mathbb{R}^n\to\mathbb{R}$ is permutation symmetric and may be random but independent of the white noise.
The function $g$ is the initial condition for equation (\ref{eq:MnBounded}) in the sense that
\[
  \lim_{t\to 0} \frac{1}{n!} \int_{\mathbb{R}^n} g(\y^\prime) Q_t(\y,\y^\prime) \;\rd{\y^\prime} = \lim_{t\to 0} \int_{W_n} g(\y^\prime) Q_t(\y,\y^\prime) \;\rd{\y^\prime} = g(\y).
\]
On the other hand, we say that $M_n(t,\x,\y)$ is the solution started from a delta initial data at $\x$ even though strictly speaking it is the ratio of $K_n(t,\x,\y)$, which can be shown to satisfy an integral equation similar to (\ref{eq:MnBounded}) with delta initial condition, and the product of Vandermonde determinants $\Delta(\x)\Delta(\y)$.

We now state the main results regarding the solutions of (\ref{eq:MnDelta}) and (\ref{eq:MnBounded}) from which Theorem \ref{thm:ZnRegularity} follows as a corollary by (\ref{eq:MnBoundary}). 
Let $\mathscr{B}_b(\mathbb{R})$ be the collection of Borel measurable subsets of $\mathbb{R}$ with finite Lebesgue measure and let $W = \big(W_t(A), t\geq 0, A\in\mathscr{B}_b(\R)\big)$ be space-time white noise on a complete probability space $(\Omega,\mathscr{F},\mathbb{P})$ endowed with a right-continuous filtration $(\mathscr{F}_t)_{t\geq 0}$ such that $W$ is $\mathscr{F}_t$-adapted and $W_t(A) - W_s(A)$ is independent of $\mathscr{F}_s$ for all $A\in\mathscr{B}_b(\R)$.
From now on we fix this filtered probability space $(\Omega,\mathscr{F},(\mathscr{F}_t)_{t\geq 0},\mathbb{P})$.
We use $\mathbb{E}$ to denote the expectation with respect to $\mathbb{P}$ and for $p\geq 1$, $\Vert\cdot\Vert_p = \left( \mathbb{E}[|\cdot|^p] \right)^{1/p}$ denotes the $L^p(\Omega)$ norm.
Throughout this paper, $c_p\leq 2\sqrt{p}$ is the constant appearing in the Burkholder--Davis--Gundy inequality. 

\begin{theorem} 
  \begin{itemize}
    \item[(a)] Suppose that $g$ is $\mathscr{F}_0$-measurable and symmetric and satisfies for all $p\geq 2$, $\sup_{\y\in\mathbb{R}^n} \Vert g(\y)\Vert_p \leq K_{p,g} < \infty$, then there exists a solution $\big(M_n^g(t,\y), (t,\y)\in [0,\infty)\times\mathbb{R}^n\big)$ to the integral equation (\ref{eq:MnBounded}) that is unique (in the sense of versions) in the class of all random fields $\big(v(t,\y), (t,\y)\in[0,\infty)\times\R^n\big)$ that satisfy $\sup_{(t,\y)\in[0,T]\times\R^n} \Vert v(t,\y)\V_p < \infty$ for all $T>0$.
        The solution satisfies for all $p\geq 2$ 
  \begin{equation} 
    \Vert M_n^g(t,\y) \Vert_p^2 < 4 K_{p,g}^2 e^{A^2 c_p^4 t},
    \label{eq:MnPthMoment}
  \end{equation}
for a constant $A>0$ depending on $n$.
  
Moreover, $M_n^g$ has a version such that $(t,\y)\mapsto M_n^g(t,\y)$ is locally H\"older continuous on $(0,\infty)\times\mathbb{R}^n$ with indices $\alpha<1/2$ in space and $\alpha<1/4$ in time.
  \item[(b)]  The chaos expansion (\ref{eq:MnChaos1})  defines a  solution $ \big(M_n(t,\x,\y), (t,\x,\y) \in(0,\infty)\times\mathbb{R}^n\times\R^n\big)$ to the integral equation (\ref{eq:MnDelta}) such that for all $p\geq 2$ and $T>0$
  \begin{equation}
    \sup_{\x,\y\in\R^n} \Vert M_n(t,\x,\y) \Vert_p^2 \leq C_{n,p,T} t^{-n^2}, \text{ for } t \leq T,
    \label{eq:MnPthMomentDelta}
  \end{equation}
  for some constant $C_{n,p,T}$.   $M_n$ is a fundamental solution to \eqref{eq:MnBounded} in the sense that whenever $g$ satisfies the assumptions of part (a) of this theorem, then, 
   \[
    M_n^g(t,\y) := \frac{1}{n!} \int_{\R^n} g(x)M_n(t,\x,\y) \Delta(x)^2 \;\rd{x}.
  \] 
  is the  unique solution to \eqref{eq:MnBounded} with initial condition $g$.

  Moreover, $M_n$ has a version such that $(t,\x,\y) \mapsto M_n(t,\x,\y)$ is locally H\"older continuous on $(0,\infty)\times \mathbb{R}^n\times\mathbb{R}^n$ with indices $\alpha<1/2$ in space and $\alpha<1/4$ in time.
  
  \end{itemize}
  \label{thm:MnMain}
\end{theorem}

\begin{theorem}
  Let $g$ be as in Theorem \ref{thm:MnMain}(a) with the additional property that $g$ is non-negative almost surely and $\mathbb{P}[g(\y)>0 \text{ for some } \y\in\R^n] = 1$.
  Then the solution $M_n^g$ to (\ref{eq:MnBounded}) satisfies
  \[
    \mathbb{P}[ M_n^g(t,\y) >0 \text{ for all } t>0 \text{ and } \y\in\mathbb{R}^n] = 1.
  \]
  Let $M_n$ be the random field defined by (\ref{eq:MnChaos1}) then
  \[
    \mathbb{P}[ M_n(t,\x,\y) >0 \text{ for all } t>0 \text{ and } \x,\y\in\mathbb{R}^n] = 1.
  \]
  \label{thm:strictPositivity}
\end{theorem}

Comparing (\ref{eq:MnDelta}) and (\ref{eq:MnBounded}) with (\ref{eq:SHEMild}), we see that they have a similar form to the mild equation of the SHE which has been well studied.
It has been shown for various initial data that the solution is H\"older continuous with indices up to $1/2$ in space and up to $1/4$ in time.
For example, the case with initial data having bounded moments was studied by Walsh in \cite{Wa86}.
Bertini and Cancrini stated the H\"older continuity in \cite{BC95} for a class of initial data which includes a delta function.
More recently, Chen and Dalang \cite{CD13b} proved the H\"older continuity for a non-linear SHE with initial data $\mu$ being a signed Borel measure over $\mathbb{R}$ such that $(|\mu|\ast p_t)(x) < \infty$ for all $t>0$ and $x\in\R$. 
For other variants of the SHE see for example \cite{CJKS14}, \cite{Sh94}, \cite{SS02} and the references therein.

In each case the tool used to prove the continuity of the solution is Kolmogorov's continuity criterion.
Denote the stochastic integral term of (\ref{eq:SHEMild}) by $I(t,y)$ then the key is to show that 
\[
  \E[|I(t,y) - I(t^\p,y^\p)|^p] \leq C\big( |y-y^\p|^{p/2} + |t-t^\p|^{p/4} \big),
\]
for $p$ large enough.
This in turn requires showing some continuity estimate for the heat kernel and in our case, estimates for the kernel $Q_t$, see Theorem \ref{thm:DysonKernelCty} below.
These estimates get increasingly involved for increasingly less regular initial data due to the $p$th moments $\E[|u(t,y)|^p]$ of the solution being unbounded as $t\downarrow 0$ or as $y\to\infty$ or both.
However for certain initial data such as a delta function, even though the $p$th moments blow up as time $t\downarrow 0$, they are for any fixed positive times uniformly bounded in space and thus one can in effect isolate the effects of the initial data by solving the equation for a small time and then start afresh with the current solution as the new initial condition.
This is the case with $M_n(t,\x,\y)$.
We will show that for all positive times $t$, $\E[|M_n(t,\x,\y)|^p]$ is bounded uniformly in space for all $p$ which puts us in the situation of (\ref{eq:MnBounded}) with $g$ having uniformly bounded $p$th moments for which continuity is easier to obtain.

The strict positivity of the solution to the stochastic heat equation was first proved by Mueller in \cite{Mu91}.
He showed that if the initial data $f$ is non-negative, continuous with compact support with $f(x)>0$ for some $x\in\mathbb{R}$, then for all $t>0$ 
\[
  \mathbb{P}[ u(t,x) > 0 \text{ for every } x\in\mathbb{R}] = 1.
\]
Bertini and Cancrini proved a weak comparison principle using the Feynman--Kac formula and used it to extend Mueller's result to a delta type initial data.
Shiga in \cite{Sh94} proved the stronger statement
\[
  \mathbb{P}[ u(t,x)>0 \text{ for every } x\in\mathbb{R} \text{ and every } t>0 ] = 1,
\]
for initial data being continuous functions such that the tails grow no faster than $e^{\lambda|x|}$ for all $\lambda>0$.
More recently, Moreno Flores in \cite{Fl14} proved the strict positivity of the solution for delta initial conditions, using a convergence result of a discrete polymer model to the SHE, see \cite{AKQ14}.
Chen and Kim \cite{CK14} further generalised the strict positivity result to the fractional SHE, which includes as a special case the SHE considered here, for measure-valued initial data by adapting Shiga's method.

In all of the proofs above (except for the polymer proof) a key result is a large deviation estimate on the stochastic integral term of the solution. 
Mueller proved such result using the fact that integrals of the type $\int_0^t \int_\mathbb{R} f(s,y) \;\W{s}{y}$ can be considered as a time-changed Brownian motion.
Chen and Kim using a method of \cite{CJK12} derived a similar estimate for the fractional SHE using Kolmogorov's continuity criterion.
We will adapt the approach of \cite{CK14} since we will first derive the necessary estimates in order to prove H\"older continuity anyway.

The outline of the paper is as follows.
In Section \ref{sec:whiteNoise} we first briefly recall integration with respect to space-time white noise and multiple stochastic integrals.
In Section \ref{sec:LpBound} we derive an upper bound on the $L^p(\Omega)$ norm of stochastic integrals which will be used repeatedly in this paper and we discuss briefly non-intersecting Brownian bridges in Section \ref{sec:nonIntersectingBM}.
We then prove some estimates on the transition density $Q_t$ in Section \ref{sec:estimatesQt} which are central to the proof of existence and continuity.
A key to the proof of the estimates is the Harish-Chandra/Itzykson--Zuber formula \cite{IZ80}.
The existence, uniqueness and moment estimates part of Theorem \ref{thm:MnMain} will be proved in Section \ref{sec:existence}.
The proof of H\"older continuity is in Section \ref{sec:cty}.
Finally, in Section \ref{sec:positivity} we prove a strong comparison principle for the integral equation (\ref{eq:MnBounded}) of which Theorem \ref{thm:strictPositivity} is a corollary.

\subsection*{Acknowledgements}
C.H.L would like to thank Roger Tribe for numerous helpful discussions on SPDEs. The research of C.H.L was supported by EPSRC grant number EP/H023364/1 through the MASDOC DTC.

\section{Preliminaries}\label{sec:prelim} 

\subsection{White Noise and Stochastic Integration}\label{sec:whiteNoise}

In this section we briefly recall the Walsh stochastic integral with respect to white noise, see for example \cite{Wa86}, \cite{Kh09} and \cite{Da99} for details.
Let $\mathscr{B}_b(\mathbb{R}^d)$ be the collection of Borel measurable subsets of $\mathbb{R}^d$ with finite Lebesgue measure.
A \emph{white noise} on $\mathbb{R}^d$ is a mean zero Gaussian random field $\{\dot{W}(A)\}_{A\in\mathscr{B}_b(\mathbb{R}^d)}$ with covariance function
\[
  \mathbb{E}[\dot{W}(A) \dot{W}(B)] = |A\cap B|, \quad\text{for all } A,B\in\mathscr{B}_b(\mathbb{R}^d),
\]
where $|\cdot|$ denotes the Lebesgue measure on $\mathbb{R}^d$.
We will only consider the case $d = 2$ and we interpret one of the dimensions as time.
More precisely, we define a \emph{space-time white noise} $\big(W_t(A), t\geq 0, A\in\mathscr{B}_b(\mathbb{R})\big)$ by $W_t(A) := \dot{W}([0,t]\times A)$ on a filtered probability space $(\Omega,\mathscr{F},(\mathscr{F}_t)_{t\geq 0},\mathbb{P})$ as described above Theorem \ref{thm:MnMain}.

A random field $f$ is elementary if it is of the form
\[
  f(s,y) = X1_{(a,b]}(s)1_A(y),
\]
where $X$ is bounded and $\mathscr{F}_a$-measurable and $A\in\mathscr{B}(\mathbb{R})$.
A simple function is a finite linear combination of elementary functions.
We say that a random field $f$ is predictable if it is measurable with respect to the $\sigma$-algebra generated by the simple functions and we say that $f\in\mathscr{P}_2$ if it is predictable and $f\in L^2(\Omega\times[0,\infty)\times\R)$.
According to Walsh's theory, \cite{Wa86}, $\{W_t(A)\}$ belongs to a suitable class of integrators called worthy martingale measures and the integral
\[
  \int_0^\infty \int_\R f(s,y) \;\W{s}{y},
\]
is defined for all $f\in\sP_2$.

We will make use of the following stochastic Fubini theorem,  \cite[Theorem 2.6]{Wa86}. 
\begin{lemma}
\label{lem:stochfubini}
Suppose that $f: \Omega \times [0,\infty) \times \R \times \R \rightarrow \R$ is predictable and that
\[
\E \left[ \int_0^\infty \int_\R \int_\R f(s,x,y)^2 \; \mu(\rd{x})\;\rd{y}\;\rd{s} \right]
\]
is finite where $\mu$ is a given finite measure on $\R$. 
Then
\[
\int_\R \left( \int_0^\infty \int_\R f(s,x,y) \;\W{s}{y}\right) \mu(\rd{x}) = \int_0^\infty \int_\R \left( \int_\R f(s,x,y) \; \mu(\rd{x})\right) \;\W{s}{y}
\]
\end{lemma}

Now we turn our attention to multiple stochastic integrals which appear in the chaos series in the introduction.
Let $k>1$.
We say that $f\in L^2_S([0,t]^k\times\mathbb{R}^k)$ if $f\in L^2([0,t]^k\times\mathbb{R}^k)$ such that $f(\pi\mb{s},\pi\mb{y}) = f(\mb{s},\mb{y})$ for all $(\mb{s},\mb{y}) \in[0,t]^k\times\mathbb{R}^k$ and $\pi\in S_k$ where $S_k$ is the set of permutations of $\{1,\ldots,k\}$ and $\pi\mb{s} = (s_{\pi 1},\ldots,s_{\pi k})$.
Let $A_1,\ldots,A_k$ be disjoint subsets of $[0,t]\times\mathbb{R}$.
An elementary function in $L_S^2([0,t]^k\times\mathbb{R}^k)$ is a function of the form 
\begin{equation}
  f(\mb{s},\mb{y}) = \sum_{\pi\in S_k} \prod_{i=1}^k 1\{(s_{\pi i},y_{\pi i}) \in A_i\}.
  \label{eq:kElementary}
\end{equation}
For such $f$ we define the $k$-fold integral by
\[
  (f\cdot W)_k(t) = \int_{[0,t]^k} \int_{\mathbb{R}^k} f(\mb{s},\mb{y}) \;\mW{k}{\mb{s}}{\mb{y}} = k! \prod_{i=1}^k \dot{W}(A_i).
\]
It can be shown that linear combinations of functions of the form (\ref{eq:kElementary}) are dense in $L_S^2([0,t]^k\times\mathbb{R}^k)$ and that for an elementary $f$, the integral $(f\cdot W)_k$ satisfies an It\^o isometry, hence for a general $f\in L_S^2([0,t]^k\times\mathbb{R}^k)$, we define $(f\cdot W)_k = \lim_{n\to\infty} (f_n\cdot W)_k$ where $\{f_n\}_{n\geq 1}$ is a sequence of elementary functions such that $f_n\to f$ in $L^2([0,t]^k\times\R^k)$.
The resulting integral is a mean zero random variable with covariance given by
\begin{equation}
  \mathbb{E}[ (f\cdot W)_k(t) (g\cdot W)_k(t)] = (f,g)_{L^2([0,t]^k\times\mathbb{R}^k)}.
  \label{eq:covarKFoldIntegral}
\end{equation}
For $f\in L^2([0,t]^k\times\mathbb{R}^k)$ that are not symmetric, we define its integral by first symmetrising $f$ via
\[
  \tilde{f}(\mb{s},\mb{y}) := \frac{1}{k!} \sum_{\pi\in S_k} f(\pi\mb{s},\pi\mb{y}),
\]
and then define
\[
  (f\cdot W)_k(t) = (\tilde{f}\cdot W)_k(t).
\]
Finally, for functions $f$ defined on $\Delta_k(t)\times\R^k$, for example the $k$-point correlation function $R_k$ appearing in (\ref{eq:ZnChaos}) and (\ref{eq:KnChaos}), we first extend it to a function on $[0,t]^k$ by setting it to be zero for $\mb{s}\notin\Delta_k(t)$ and then define
\[
  \int_{\Delta_k(t)} \int_{\mathbb{R}^k} f(\mb{s},\mb{y}) \;\mW{k}{\mb{s}}{\mb{y}} := (\tilde{f}\cdot W)_k(t).
\]

Now define a time reversed white noise $\tilde{W}$ by $\tilde{W}([0,s]\times A) = \dot{W}([t-s,t]\times A)$, $s\leq t$ and $A\in\mathscr{B}_b(\mathbb{R})$.
We will need the following result for the proof of continuity in Section \ref{sec:ctyInitialData}.

\begin{lemma}
  Let $f\in L_S^2([0,t]^k\times\mathbb{R}^k)$ then
  \[
    \int_{[0,t]^k} \int_{\mathbb{R}^k} f(\mb{s},\mb{y}) \;\mW{k}{\mb{s}}{\mb{y}} = \int_{[0,t]^k} \int_{\mathbb{R}^k} f(t-\mb{s},\mb{y}) \;\tilde{W}^{\otimes k}(\rd{\mb{s}},\rd{\mb{y}}) \quad \mathrm{a.s.,}
  \]
  where $t-\mb{s} = (t-s_1,\ldots,t-s_k)$.
  \label{lem:timeReversal}
\end{lemma}

\begin{proof}
  The result in the case when $f$ is an elementary function of the form (\ref{eq:kElementary}) follows from the definition of the integral and the definition of $\tilde{W}$.
  For general $f\in L_S^2([0,t]^k\times\mathbb{R}^k)$, let $\{f_n\}_{n\geq 1}$ be a sequence of elementary functions converging to $f$. 
  The result of the lemma holds for $(f_n\cdot W)_k(t)$ for all $n$ and by taking limits we see that the result also holds for $(f\cdot W)_k(t)$.
\end{proof}

\subsection{\texorpdfstring{$L^p$}{Lp} Bounds on Stochastic Integrals}\label{sec:LpBound}

The following estimate is a useful bound on the $L^p(\Omega)$ norm of stochastic integrals; it can be considered as a version of \cite[Lemma 2.4]{CK12} or \cite[Lemma 3.3]{FK09} adapted to the present setting. 
Recall that for brevity we denote $\rd{\y_*^\prime} = \rd{y_2^\p}\ldots\rd{y_n^\p}$ and $c_p\leq 2\sqrt{p}$ is the constant appearing in the Burkholder--Davis--Gundy inequality.

\begin{lemma}
  Define a random field $\big( f(t,\y); (t,\y)\in(0,\infty)\times\mathbb{R}^n \big)$ by
  \[
    f(t,\y) = \int_0^t \int_{\mathbb{R}^n} \Gamma_{t-s}(\y,\y^\prime) w(s,\y^\prime) \;\rd{\y_*^\prime} \:\W{s}{y_1^\prime},
  \]
  for a suitable random field $w$ and $\Gamma_t(\y,\y^\prime)$ is a non-random and non-negative measurable function on $(0,\infty)\times\mathbb{R}^n\times\mathbb{R}^n$ such that $\int_{\mathbb{R}^{n-1}} \Gamma_{t-s}(\y,\y^\prime) w(s,\y^\prime) \;\rd{\y_*^\prime}$ is integrable in the sense of Walsh for all $(t,\y)\in(0,\infty)\times\mathbb{R}^n$.
  Then for all integers $p\geq 2$, $t\geq 0$ and $\y\in\mathbb{R}^n$
  \begin{align*}
    \Vert f(t,\y) \Vert_p^2 \leq c_p^2 \int_0^t \int_\mathbb{R} \bigg( \int_{\mathbb{R}^{n-1}} \Gamma_{t-s}(\y,\y^\prime) \Vert w(s,\y^\prime) \Vert_p \;\rd{\y_*^\prime} \bigg)^2 \mathrm{d}y_1^\prime \mathrm{d}s.
  \end{align*}
  \label{lem:integral}
\end{lemma}

\begin{proof}
  Fix $t$ and $\y$, then by the Burkholder--Davis--Gundy inequality applied to the martingale $\big( \int_0^r \int_{\R^n} \Gamma_{t-s}(\y,\y^\p) w(s,\y^\p) \;\rd{\y_*^\p} \:\W{s}{y_1^\p}$, $r\in [0,t] \big)$, we have
  \[
    \V f(t,\y) \Vert_p^2 \leq c_p^2 \bigg\V \int_0^t \int_\R \bigg( \int_{\R^{n-1}} \Gamma_{t-s}(\y,\y^\prime) w(s,\y^\prime) \;\rd{\y_*^\prime} \bigg)^2 \rd{y_1^\prime} \rd{s} \bigg\V_{p/2}.
  \]
  Applying Minkowski's integral inequality \cite[Corollary 1.30]{Kal02} twice, we obtain 
  \begin{align*}
    \V f(t,\y) \V_p^2 
    &\leq c_p^2 \int_0^t \int_\R \bigg\V \int_{\R^{n-1}} \Gamma_{t-s}(\y,\y^\prime) w(s,\y^\prime) \;\rd{\y_*^\prime} \bigg\V_p^2 \rd{y_1^\prime}\rd{s} \\
    &\leq c_p^2 \int_0^t \int_\R \bigg(\int_{\R^{n-1}} \Gamma_{t-s}(\y,\y^\p) \V w(s,\y^\p) \V_p \;\rd{\y_*^\p} \bigg)^2 \rd{y_1^\p}\rd{s},
  \end{align*}
  as required.
\end{proof}

\begin{lemma}
  For all $k\geq 1$ and $f\in L^2(\Delta_k(t)\times\R^k)$ we have
  \[
    \bigg\V \int_{\Delta_k(t)} \int_{\R^k} f(\mb{s},\mb{y}) \;\mW{k}{\mb{s}}{\mb{y}} \bigg\V_p^2 \leq c_p^{2k} \int_{\Delta_k(t)} \int_{\R^k} f(\mb{s},\y)^2 \;\rd{\y}\rd{\mb{s}}.
  \]
  \label{lem:LpChaos}
\end{lemma}

\begin{proof}
  Since multiple stochastic integrals on $\Delta_k(t)$ coincides with iterated stochastic integrals, applying Burkholder--Davis--Gundy inequality and Minkowski's integral inequality $k$ times gives the desired upper bound.
\end{proof}

\subsection{Predictability of Random Fields}

Recall that the Walsh integral is defined for random fields in $\mathscr{P}_2$, see Section \ref{sec:whiteNoise} above, therefore it is convenient to have a set of conditions to verify the predictability of a random field.
The following result is from \cite[Proposition 3.1]{CD13} which is an extension of \cite[Proposition 2]{DF98} to space-time white noise.

\begin{proposition}\label{prop:predictability}
  Let $t>0$ and suppose a random field $\big(f(s,y), (s,y)\in (0,t)\times\mathbb{R}\big)$ satisfies
  \begin{itemize}
    \item[(i)] $f$ is adapted, that is for all $(s,y) \in (0,t)\times\mathbb{R}$, $f(s,y)$ is $\mathscr{F}_s$-measurable;
    \item[(ii)] for all $(s,y)\in(0,t)\times\mathbb{R}$, $\Vert f(s,y)\Vert_2 < \infty$ and $(s,y)\mapsto f(s,y)$ is $L^2(\Omega)$-continuous on $(0,t)\times\mathbb{R}$;
    \item[(iii)] $\int_0^t\int_\mathbb{R} \Vert f(s,y) \Vert_2^2 \;\dd{y}{s} < \infty$.
  \end{itemize}
  Then $f\in\mathscr{P}_2$ and  
  \[
    \int_0^t \int_\mathbb{R} f(s,y) \;\W{s}{y},
  \]
  is a well-defined Walsh integral.
\end{proposition}

In the sequel we will need to integrate functions of the form: for some random field $M$, let $f(s,y_1^\prime) = \int_{\mathbb{R}^{n-1}} Q_{t-s}(\y,\y^\prime) M(s,\y^\prime) \;\rd{\y_*^\prime}$ where we recall $Q_t$ is the transition density of Dyson Brownian motion defined in the introduction. (Note that we have suppressed the dependency of $f$ on $t$ and $\y$ to keep the notation simple).
The following proposition provides convenient conditions to verify the integrability of such a random field.

\begin{proposition}
  Let $t>0$ and $\y\in\mathbb{R}^n$.
  Suppose the random field $\big(M(s,\y^\prime), (s,\y^\prime)\in(0,t)\times\mathbb{R}^n\big)$ satisfies
  \begin{itemize}
    \item[(i)] $M$ is adapted i.e., for all $(s,\y^\prime)\in(0,t)\times\mathbb{R}^n$, $M(s,\y^\prime)$ is $\mathscr{F}_s$-measurable;
    \item[(ii)] $(s,\y^\prime) \mapsto M(s,\y^\prime)$ is $L^2(\Omega)$-continuous on $(0,t)\times\mathbb{R}^n$;
    \item[(iii)] $\sup_{(s,\y^\prime)\in(0,t)\times\mathbb{R}^n} \Vert M(s,\y^\prime) \Vert_2 < \infty$.
 \end{itemize}
 Then $\big(f(s,z), (s,z)\in(0,t)\times\mathbb{R}\big)$ defined by $f(s,y_1^\prime) = \int_{\mathbb{R}^{n-1}} Q_{t-s}(\y,\y^\prime) M(s,\y^\prime) \;\rd{\y_*^\prime}$ is in $\mathscr{P}_2$ and 
 \[
  \int_0^t \int_\mathbb{R} f(s,y_1^\prime) \;\W{s}{y_1^\prime},
 \]
  is a well-defined Walsh integral.
  \label{prop:predictability2}
\end{proposition}

\begin{proof}
  We will show that $f$ satisfies the three assumptions of Proposition \ref{prop:predictability}.
  Since $Q_{t-s}(\y,\y^\prime)$ is continuous and deterministic, $Q_{t-s}(\y,\y^\prime) M(s,\y^\prime)$ is adapted by (i) and so the integral $\int_{\mathbb{R}^{n-1}} Q_{t-s}(\y,\y^\prime) M(s,\y^\prime) \;\rd{\y_*^\prime}$ is also adapted. 
  Assumption (iii) of Proposition \ref{prop:predictability} follows from (iii) above since by Minkowski's integral inequality and Lemma \ref{lem:integralofKSquared} below, we have for some constant $C$
  \begin{align*}
    \int_0^t \int_\mathbb{R} \Vert f(s,y_1^\prime) \Vert_2^2 & \;\rd{y_1^\prime} \rd{s} \\
    &\leq \int_0^t \int_\mathbb{R} \bigg( \int_{\mathbb{R}^{n-1}} Q_{t-s}(\y,\y^\prime) \Vert M(s,\y^\prime) \Vert_2 \;\rd{\y_*^\prime} \bigg)^2 \rd{y_1^\prime} \rd{s} \\
    &\leq \sup_{(s,\y^\prime)\in(0,t)\times\mathbb{R}^n} \Vert M(s,\y^\prime) \Vert_2^2 \int_0^t \int_\mathbb{R} \bigg( \int_{\mathbb{R}^{n-1}} Q_{t-s}(\y,\y^\prime) \;\rd{\y_*^\prime} \bigg)^2 \rd{y_1^\prime}\rd{s} \\
    &\leq 2C t^{1/2} \sup_{(s,\y^\prime)\in(0,t)\times\mathbb{R}^n} \Vert M(s,\y^\prime) \Vert_2^2.
  \end{align*}
  It remains to show the $L^2(\Omega)$-continuity of $f$.
  We wish to show that for each $(s,y^\prime) \in(0,t)\times\mathbb{R}$, $\lim_{(u,z)\to(s,y^\prime)} \Vert f(u,z) - f(s,y^\prime) \Vert_2 = 0$.
  Let $u\in[s/2,(t+s)/2]$ then by Lemma \ref{lem:QtGaussianBound} below we have for a constant $C$ depending on $n$ that
  \begin{align*}
    Q_{t-u}(\y,\mb{z}) 
      &\leq C(t-u)^{-n/2} \prod_{i=1}^n e^{y_i^2/2(t-u)} e^{-z_i^2/8(t-u)} \\
      &\leq \frac{2^{n/2} C}{(t-s)^{n/2}} \prod_{i=1}^n e^{y_i^2/(t-s)} \prod_{i\neq 1} e^{-z_i^2/8(t-s/2)}.
  \end{align*}
  The last line is integrable with respect to $\rd{\mb{z}_*} = \rd{z_2}\ldots\rd{z_n}$ and so by the dominated convergence theorem, the continuity of $Q_t$ and assumption (ii), the right hand side of 
  \begin{align*}
    \Vert & f(u,z_1) - f(s,y_1^\prime) \Vert_2 \\
    &\leq \sup_{(u,\y) \in [s/2, (t+s)/2] \times \R^n} \Vert M(u,\y) \Vert_2 \int_{\mathbb{R}^{n-1}} \big|Q_{t-u}\big(\y,(z_1,\mb{z}_*)\big) - Q_{t-s}\big(\y,(y_1^\prime,\mb{z}_*)\big)\big| \;\rd{\mb{z}_*} \notag \\
    &\qquad+ \int_{\R^{n-1}} Q_{t-s}\big(\y,(y_1^\prime,\mb{z}_*)\big) \V M\big(u,(z_1,\mb{z}_*)\big) - M\big(s,(y_1^\p,\mb{z}_*)\big) \V_2 \;\db{z_*}
  \end{align*}
  converges to zero as $(u,z_1) \to (s,y_1^\prime)$. 
  Finally, an application of Proposition \ref{prop:predictability} completes the proof.
\end{proof}

\subsection{Non-intersecting Brownian Motions}\label{sec:nonIntersectingBM}

Dyson Brownian motion introduced in \cite{Dy62} can be realised as the eigenvalues of Hermitian Brownian motion, an $n\times n$ Hermitian matrix whose entries are (up to the Hermitian condition) independent standard complex Brownian motions. 
The eigenvalues of such a matrix is a Markov process with state space $W_n$ with transition density $Q_t(\x,\y)$. 
It also arises as the Doob $h$-transform of Brownian motion killed at the boundary $\partial W_n$ with $h(\x) = \Delta(\x)$ (see for example \cite{Gr99} and \cite{KT07}).

One can construct bridges of Dyson Brownian motion, which we will call Dyson Brownian bridge or non-intersecting Brownian bridges, using the framework of \cite{FPY92}.
For $\x$, $\y\in W_n$, a collection of non-intersecting Brownian bridges $X_s = (X_s^1,\ldots,X_s^n)$, $0\leq s\leq t$, starting at $\x$ at time 0 and ending at $\y$ at time $t$ is a process whose law is absolutely continuous on $\sigma(X_u; u\leq s)$ for any $s<t$ to that of Dyson Brownian motion started at $\x$ with Radon--Nikodym derivative equal to
\[
 \frac{Q_{t-s}(X_s,\y)}{Q_t(\x,\y)}.
\]
In particular, for $0 < s_1 < \ldots < s_k < t$, the law of $(X_{s_1},\ldots,X_{s_k})$ is given by the density
\begin{align*}
  \frac{Q_{s_1}(\x,\y^1) \prod_{i=2}^k Q_{s_i-s_{i-1}}(\y^{i-1},\y^i) Q_{t-s_k}(\y^k,\y)}{Q_t(\x,\y)} 
\end{align*}
The above is well defined at the boundary of the Weyl chamber, see Section 2 of \cite{OW11}; in particular by (\ref{eq:QtBoundary}), taking limits as $\x\to a\1$, $\y\to b\1$ where $\1 = (1,\ldots,1)$ one obtains
\begin{align*}
  c_n \frac{\Delta(\y^1) \Delta(\y^k) \prod_{j=1}^n p_{s_1}(a-y_j^1) \prod_{i=2}^k p_n^*(s_i-s_{i-1},\y^{i-1},\y^i)  \prod_{j=1}^n p_{t-s_k}(b-y_j^k)}{s_1^{n(n-1)/2} (t-s_k)^{n(n-1)/2} t^{-n(n-1)/2} p_t(a-b)^n},
\end{align*}
where $c_n^{-1} = \prod_{i=1}^{n-1} i!$. 
The $k$-point correlation function $R_k(\mb{s},\y_1;t,\x,\y)$ with $\y_1 = (y_1^1,\ldots,y_1^k)$ appearing in (\ref{eq:KnChaos}) is defined as
\begin{align}
 \bigl ((n-1)!\bigr)^{-k} \int_{(\mathbb{R}^{n-1})^k} \frac{Q_{s_1}(\x,\y^1) \prod_{i=2}^k Q_{s_i-s_{i-1}}(\y^{i-1},\y^i) Q_{t-s_k}(\y^k,\y)}{Q_t(\x,\y)} \;\prod_{i=1}^k\prod_{j=2}^n \rd{y_j^i}.
  \label{eq:Rk}
\end{align}
For each $i$ we have chosen to leave the first coordinate of $\y^i$ and integrated out the rest but this choice is arbitrary by symmetry. 
Note that this is also the reason for the form of the stochastic integral term in (\ref{eq:MnDelta}).

In the sequel we will need to bound integrals of the square of the $k$-point correlation function $R_k$.
Correlation functions of densities given by a product of determinants have been studied extensively in the context of determinantal point processes, see for example \cite{Joh06} and \cite{Bo11}.
They can be expressed as a determinant of a matrix whose entries are given by some kernel function.
However for general start and end points $\x$ and $\y$ this kernel function is difficult to compute, but since all we need is the integral of the square of $R_k$ it is not necessary to compute $R_k$ explicitly and so we will not pursue this.
Instead, the next two results proved in \cite{OW11} which expresses the integral of $R_k^2$ in terms of intersection local times of Brownian bridges will be used. 
Let $X = (X^1,\ldots,X^n)$ and $Y = (Y^1,\ldots,Y^n)$ be two independent copies of a collection of $n$ non-intersecting Brownian bridges which start at $\x$ at time 0 and end at $\y$ at time $t$ and let $\E_{\x,\y;t}^{X,Y}$ denote the corresponding expectation of the joint law of the bridges.
Let $L_t(X^i-Y^j)$ be the local time at 0 of the difference $X^i-Y^j$.
Then we have, see \cite[Lemma 4.1]{OW11},
\begin{lemma}
Fix $n\geq 1$. For all integers $k\geq 1$ and all $t>0$, $\x$, $\y\in W_n$ the following holds
  \[
    \int_{\Delta_k({t})}\int_{\R^k} R_k(\mb{s},\y^\prime; t,\x,\y)^2 \;\db{y}^\prime\db{s} = \frac{1}{k!} \E_{\x,\y;t}^{X,Y}\Big[\Big( \sum_{i,j=1}^n L_t(X^i-Y^j) \Big)^k \Big].
  \]
  \label{lem:localTimeFormula}
\end{lemma}
The following is used to bound the above moments of local times. 

\begin{lemma}
  For all $a\geq 1$ and $0<t\leq T$, there exists  constants $C := C(a,n,T)$  and  $C^\prime:= C^\prime(a,n,T)$ such that
  \[
    \sup_{\x,\y\in W_n} \bigg(\frac{p_n^*(t,\x,\y)}{\Delta(\x)\Delta(\y)}\bigg)^2 \E_{\x,\y;t}^{X,Y}\Big[ \exp\Big(a\! \sum_{i,j=1}^n L_t(X^i-Y^j) \Big)\Big] \leq C t^{-n^2/2} \frac{p_n^*(t,\x,\y)}{\Delta(\x)\Delta(\y)}  \leq C^\prime t^{-n^2}.
  \]
  \label{lem:localTimeExpMoments}
\end{lemma}
The above two lemmata shows that for each $t>0$, $\sup_{x,y} \V Z_n(t,x,y) \V_2 < \infty$ and thus the chaos series (\ref{eq:ZnChaos}) is convergent in $L^2(\Omega)$. 
The same is also true for (\ref{eq:KnChaos}).

\begin{proof}[ Proof of Lemma \ref{lem:localTimeExpMoments}]

The first inequality is essentially Proposition 5.2 in \cite{OW11}, however the proof given there contains an error which we correct here.

By writing 
\[
\sum_{i,j=1}^n L_t(X^i-Y^j) = \sum_{i,j=1}^n L_{t/2}(X^i-Y^j)+ \sum_{i,j=1}^n \bigl(L_t(X^i-Y^j)-  L_{t/2}(X^i-Y^j)\bigr),
\]
and applying the Cauchy--Schwarz inequality, together with the time reversibility of the non-intersecting Brownian bridges, we see that it is enough to show that, for all $\x,\y\in W_n$,
\begin{equation} \label{eq:intlocaltimebound}
 \frac{p_n^*(t,\x,\y)}{\Delta(\x)\Delta(\y)}\E_{\x,\y;t}^{X,Y}\Big[ \exp\Big( 2a \sum_{i,j=1}^n L_{t/2}(X^i-Y^j) \Big)\Big]  \leq C t^{-n^2/2}.
\end{equation}
The law of the pair of systems of non-intersecting bridges over the time interval $[0,t/2]$ is absolutely continuous with respect to the law of a pair of independent Dyson  Brownian motions, with Radon--Nikodym derivative,
\[
 \frac{Q_{t/2}(X_{t/2},\y)}{Q_t(\x,\y)}  \frac{Q_{t/2}(Y_{t/2},\y)}{Q_t(\x,\y)}. 
\]
The error in \cite{OW11} is the omission of the second factor from this density.
An application of Cauchy--Schwarz, together with the fact that $X$ and $Y$ are independent and identically distributed gives,
\begin{multline*}
\E_{\x,\y;t}^{X,Y}\Big[ \exp\Big( 2a \sum_{i,j=1}^n L_{t/2}(X^i-Y^j) \Big)\Big]  \leq  \\
\hat{\E}_{\x}^{X}\left[  \frac{Q_{t/2}(X_{t/2},\y)^2}{Q_t(\x,\y)^2}\right]
\times\left( \hat{\E}_{\x}^{X,Y}\Big[ \exp\Big( 4a \sum_{i,j=1}^n L_{t/2}(X^i-Y^j) \Big)\Big]  \right)^{1/2},
\end{multline*}
where $\hat{\E}^X_{\x}$ denotes the expectation with respect to which the law of $X$ is a Dyson Brownian motion starting from $\x$, and $\hat{\E}^{X,Y}_{\x}$ denotes the expectation with respect to which the law of $X$ and $Y$ is that of a pair of independent Dyson Brownian motions, each starting from $\x$. On the righthand side of the above inequality the second factor is controlled by the argument of Proposition 4.2 of \cite{OW11}, and it is bounded by a constant independently of $\x \in W_n$ and $t \in [0,T]$. To control the first factor we observe that the Harish-Chandra formula  \eqref{eq:HCIZ} implies that
\begin{equation} \label{eq:HCbound}
 \frac{Q_{t}(\x,\y)}{\Delta(\y)^2}= \frac{p_n^*(t,\x,\y)}{\Delta(\x)\Delta(\y)} \leq K_{n} t^{-n^2/2}
 \end{equation}
 for a constant $K_n$ depending on $n$ only. Consequently, bounding a single factor of $Q_{t/2}(X_{t/2},\y)$ using this inequality, gives, 
 \begin{align*}
 \hat{\E}_{\x}^{X}\left[  \frac{Q_{t/2}(X_{t/2},\y)^2}{Q_t(\x,\y)^2}\right]   &\leq    K_{n} (t/2)^{-n^2/2}\frac{\Delta(\x)\Delta(\y)}{p_n^*(t,\x,\y)}  \hat{\E}_{\x}^{X}\left[  \frac{Q_{t/2}(X_{t/2},\y)}{Q_t(\x,\y)}\right]   \\&=   K_{n} (t/2)^{-n^2/2}\frac{\Delta(\x)\Delta(\y)}{p_n^*(t,\x,\y)},
 \end{align*}
 when we note  that  $\frac{Q_{t/2}(X_{t/2},\y)}{Q_t(\x,\y)}$ is the Radon--Nikodym density of a probability measure and therefore has expectation one. This proves \eqref{eq:intlocaltimebound}, and hence the first inequality in the statement of the lemma, with
 \[
 C= K_{n} 2^{n^2/2} \sup_{\x\in W_n} \left( \hat{\E}_{\x}^{X,Y}\Big[ \exp\Big( 4a \sum_{i,j=1}^n L_{t/2}(X^i-Y^j) \Big)\Big]  \right)^{1/2}
 \]
 Finally the second inequality in the statement follows from the first on a further application of \eqref{eq:HCbound}. 
\end{proof}

\section{Estimates on \texorpdfstring{$Q_t$}{Qt}}\label{sec:estimatesQt}


Before proving Theorem \ref{thm:MnMain} we need estimates on various quantities involving the kernel $Q_t(x,y) = \frac{\Delta(\y)}{\Delta(\x)}(2\pi t)^{-n/2}\det[e^{-(x_i-y_j)^2/2t}]_{i,j=1}^n$. 
The following known as the Harish-Chandra/Itzykson--Zuber formula \cite{IZ80} provides a useful alternate expression for $Q_t$:
\begin{equation}\label{eq:HCIZ}
  \frac{\det[e^{-(x_i-y_j)^2/2t}]_{i,j=1}^n}{\Delta(\x)\Delta(\y)} = c_n t^{-n(n-1)/2} \int_{\mathcal{U}(n)} \exp\Big(-\frac{1}{2t}\text{Tr} (D_\y - UD_\x U^\dagger)^2 \Big) \;\mathrm{d}U,
 \end{equation}
where $D_\x$ and $D_\y$ are diagonal matrices with entries $x_1,\ldots,x_n$ and $y_1,\ldots,y_n$ respectively.
$c_n = \big(\prod_{i=1}^{n-1} i!\big)^{-1}$ and the integral is with respect to the normalised Haar measure on the unitary group $\mathcal{U}(n)$.
Furthermore, the integrand above is bounded uniformly in $U$ as the following bound from \cite[Lemma 1]{MRTZ06} shows
\begin{equation}
  \sup_{U\in\mathcal{U}(n)} \exp\Big(-\frac{1}{2t}\text{Tr}(Y-UXU^\dagger)^2\Big) \leq \prod_{i=1}^n e^{-(y_i-x_i)^2/2t}.
  \label{eq:supHCIntegrand}
\end{equation}
From this and the Harish-Chandra formula we also have the following
\begin{lemma}
  For all $t>0$ and $x$, $y\in\R^n$ we have with a constant $C$ depending on $n$,
  \[
    Q_t(\x,\y) \leq C t^{-n/2} \prod_{i=1}^n e^{x_i^2/2t} e^{-y_i^2/8t}.
  \]
  \label{lem:QtGaussianBound}
\end{lemma}

\begin{proof}
  First note that
  \begin{align*}
    e^{-(y_i-x_i)^2/2t} = e^{-(y_i^2 - 2x_i^2)/4t} e^{-(y_i - 2x_i)^2/4t} \leq e^{-(y_i^2 - 2x_i^2)/4t},
  \end{align*}
  and so by the Harish-Chandra formula and \eqref{eq:supHCIntegrand}, recalling the definition of $Q_t$ \eqref{eq:QtDefn}, we have
  \[
    Q_t(\x,\y) \leq \frac{c_n}{(2\pi)^{n/2} t^{n^2/2}} \Delta(\y)^2 \prod_{i=1}^n e^{-y_i^2/4t} e^{x_i^2/2t}.
  \]
  The Vandermonde determinant $\Delta(\y)$ is a homogeneous polynomial in $y_1,\ldots,y_n$ of degree $n(n-1)/2$ and since $\sup_y y^\alpha e^{-y^2/\beta} \leq C_\alpha \beta^{\alpha/2}$, $\alpha, \beta>0$ where $C_\alpha$ is a constant depending on $\alpha$, we have for a constant $C$ depending on $n$ 
  \[
    \Delta(\y)^2 \prod_{i=1}^n e^{-y_i^2/8t} \leq C t^{n(n-1)/2},
  \]
  and from which the statement of the lemma follows.
\end{proof}

As mentioned in the introduction, $Q_t(\x,\y)$ is well defined on the boundary of the Weyl chamber and since it is a product and ratio of determinants, it is permutation symmetric and so we can extend $Q_t$ to a function on $\R^n\times\R^n$ by symmetry.
Denote $K_t(\x,y_1) := \int_{\mathbb{R}^{n-1}} Q_t(\x,\y) \;\rd{\y_*}$, recalling that $\rd{\y}_*$ means integration with respect to $y_2,\ldots,y_n$ for $\y = (y_1,\ldots,y_n)\in\R^n$.
The following result strongly indicates the continuity of $M_n$; in fact it is a key estimate in its proof in Section \ref{sec:cty}.

\begin{theorem}
  \phantomsection\label{thm:DysonKernelCty}
  \begin{itemize}
    \item[(a)] There is a constant $C_1>0$ depending only on $n$ such that for all $t>0$ and $\x$, $\mb{z}\in\mathbb{R}^n$ we have
    \[
      \int_0^t \int_\mathbb{R} \left( \int_{\R^{n-1}} | Q_{s}(\x,\y) - Q_{s}(\mb{z},\y) | \;\rd{\y_*} \right)^2 \;\rd{y_1}\rd{s} \leq C_1|\x-\mb{z}|,
    \]
    \item[(b)] there are constants $C_2$, $C_3>0$ depending only on $n$ such that for all $t$, $u$ with $0<u\leq t<\infty$ and $\x\in\mathbb{R}^n$, we have
    \[
      \int_0^u \int_\mathbb{R} \left( \int_{\R^{n-1}} | Q_{t-u+s}(\x,\y) - Q_{s}(\x,\y) | \;\rd{\y_*} \right)^2 \;\rd{y_1}\rd{s} \leq C_2|t-u|^{1/2},
    \]
    and
    \[
      \int_u^t \int_\mathbb{R} K_{s}(\x,y)^2 \;\rd{y}\rd{s} \leq C_3|t-u|^{1/2}.
    \]
  \end{itemize}
\end{theorem}

The theorem is a consequence of the series of results below.
First observe that $Q_t$ has the following scaling property:
\begin{align}
  Q_t(\x,\y) 
  = t^{-n/2} \frac{\Delta(\y/\sqrt{t})}{\Delta(\x/\sqrt{t})} \det\Big[\frac{1}{\sqrt{2\pi}} e^{-(x_i/\sqrt{t} - y_j/\sqrt{t})^2/2} \Big] 
  = t^{-n/2} Q_1(\x/\sqrt{t},\y/\sqrt{t}).
  \label{eq:QtScaling}
\end{align}
The left hand side of the inequality in Theorem \ref{thm:DysonKernelCty}(a) is bounded above by
\begin{align}
  \int_0^\infty & \int_\mathbb{R} \bigg( \int_{\mathbb{R}^{n-1}} | Q_s(\x,\y) - Q_s(\z,\y) | \;\rd{\y_*} \bigg)^2 \;\mathrm{d}y_1 \mathrm{d}s \notag \\
  &= \int_0^\infty \frac{1}{\sqrt{s}} \int_\mathbb{R} \bigg( \int_{\mathbb{R}^{n-1}} | Q_1(\x/\sqrt{s}, \y^\prime) - Q_1(\z/\sqrt{s},\y^\prime) | \;\rd{\y_*^\prime} \bigg)^2 \rd{y_1^\prime} \rd{s},
  \label{eq:QtScaling2}
\end{align}
where we have changed the integration region to $[0,\infty)$ in the time integral which results in an upper bound due to the positivity of the integrand.
The equality follows from the scaling property (\ref{eq:QtScaling}) and a change of variables. 
Theorem \ref{thm:DysonKernelCty}(a) follows from (\ref{eq:QtScaling2}) and Lemma \ref{lem:min} below.
\begin{lemma}
Suppose a function $R(\x,\z;y):\mathbb{R}^n\times\R^n\times\mathbb{R}\to\mathbb{R}$ satisfies for some constants $c_1$, $c_2>0$
\begin{equation}
  \int_\mathbb{R} R(\x,\z;y)^2 \;\mathrm{d}y \leq \min(c_1,c_2|\x-\z|^2),
  \label{eq:min}
\end{equation}
for any $\x$, $\z\in\mathbb{R}^n$, then
\[
 \int_0^\infty \frac{1}{\sqrt{t}} \int_\mathbb{R} R(\x/\sqrt{t},\z/\sqrt{t};y)^2 \;\mathrm{d}y\mathrm{d}t \leq C|\x-\z|,
\]
with $C = 4\sqrt{c_1c_2}$.
  \label{lem:min}
\end{lemma}

\begin{proof}
\begin{align*}
 \int_0^\infty \frac{1}{\sqrt{t}} \int_\mathbb{R} & R(\x/\sqrt{t},\z/\sqrt{t};y)^2 \;\mathrm{d}y \mathrm{d}t \\
 	&\leq \int_0^{\frac{c_2}{c_1}|\x-\z|^2} \frac{c_1}{\sqrt{t}} \;\mathrm{d}t + \int_{\frac{c_2}{c_1}|\x-\z|^2}^\infty \frac{c_2}{t^{3/2}}|\x-\z|^2 \;\mathrm{d}t = C|\x-\z|.
\end{align*}
\end{proof}

Thus, we need to show that
\[
  R(\x,\z;y_1) := \int_{\R^{n-1}} | Q_1(\x,\y) - Q_1(\z,\y) | \;\rd{\y_*}
\]
satisfies the hypothesis of Lemma \ref{lem:min}.
Using the inequality $(a+b)^2 \leq 2(a^2 + b^2)$ and the fact that $Q_t$ is non-negative, we have
\begin{align*}
    \int_\R R(\x,\z;y_1)^2 \;\rd{y_1} 
    &\leq 2\bigg(\int_\mathbb{R} K_1(\x,y_1)^2 \;\mathrm{d}y_1 + \int_\mathbb{R} K_1(\z,y_1)^2 \;\mathrm{d}y_1 \bigg) \\
    &\leq 4\sup_{\x\in\mathbb{R}^n} \Vert K_1(\x,\cdot) \Vert_{L^2(\mathrm{d}y_1)}^2.
\end{align*}
On the other hand, let $\mb{r}(\rho):[0,1]\to\mathbb{R}^n$, $\mb{r}(\rho) = (1-\rho)\x + \rho \z$ be a parameterisation of the straight line from $\x$ to $\z$ and denote by $\nabla Q_1$ the gradient of $Q_1$ with respect to the first variable then 
\begin{align*}
  \int_{\R^{n-1}} | Q_1(\x,\y) - Q_1(\z,\y) | \;\rd{\y_*}
  &\leq \int_{\R^{n-1}} \int_0^1 | \nabla Q_1(\mb{r}(\rho),\y) \cdot \mb{r}^\prime(\rho) | \;\mathrm{d}\rho \rd{\y_*} \\
  &\leq \int_{\R^{n-1}} \int_0^1 | \nabla Q_1(\mb{r}(\rho),\y) | |\x-\z| \;\rd{\rho}\rd{\y_*}
\end{align*}
By Lemma \ref{prop:L1Derivative} below we have $\left|\frac{\partial Q_1}{\partial x_j}(\x,\y) \right| \leq C Q_2(\x,\y)$ for all $j$ and so 
\[
  |\nabla Q_1(\x,\y)|^2 = \sum_{j=1}^n \frac{\partial Q_1}{\partial x_j} (\x,\y)^2 \leq nC^2 Q_2(\x,\y)^2.
\]
Thus, by Minkowski's integral inequality
\begin{align*}
  \bigg(\int_\R R(\x,\z;y_1)^2 \;\rd{y_1} \bigg)^{1/2}
  &\leq C^\p |\x-\z| \int_0^1 \bigg( \int_\R K_2(\mb{r}(\rho),y_1)^2 \;\rd{y_1} \bigg)^{1/2} \;\rd{\rho} \\
  &\leq C^\p |\x-\z| \sup_{\x\in\R^n} \Vert K_2(\x,\cdot) \Vert_{L^2(\mathrm{d}y_1)},
\end{align*}
for a constant $C^\p$ depending only on $n$.
Therefore, in order to verify the hypothesis of Lemma \ref{lem:min} it suffices by scaling to show that
\begin{equation}\label{eq:L2}
 \sup_{\x\in\mathbb{R}^n} \int_\mathbb{R} K_1(\x,y)^2 \;\mathrm{d}y < \infty.
\end{equation}

\begin{lemma}\label{prop:L1Derivative}
  There is a constant $C := C(n)$ such that for all $j=1,\ldots,n$,   
  \[
    \bigg\vert \frac{\partial Q_t}{\partial x_j}(\x,\y) \bigg\vert \leq CQ_{2t}(\x,\y).
  \]
\end{lemma} 

\begin{proof}
 By the Harish--Chandra formula \eqref{eq:HCIZ}, $Q_t(\x,\y)$ can be written as
 \begin{align*}
  Q_t(\x,\y) 
  = c_n^\prime t^{-n^2/2} \int_{\mathcal{U}(n)} \Delta(\y)^2 \exp\Big(-\frac{1}{2t}\text{Tr}(D_\y - UD_\x U^\dagger)^2\Big) \;\mathrm{d}U,  
 \end{align*}
 where $c_n^\prime = (2\pi)^{-n/2} c_n$.
 Observe that by the cyclic property of the trace and the fact that $U$ is unitary, $\text{Tr}(D_\y-UD_\x U^\dagger)^2 = \text{Tr}(U^\dagger D_\y U -D_\x)^2$. 
  Therefore,
 \begin{equation}\label{eq:dQ1}
     \frac{\partial Q_t}{\partial x_j}(\x,\y) = c_n^\prime t^{-n^2/2} \int_{\mathcal{U}(n)} \frac{1}{t} \Delta(\y)^2 \big((U^\dagger D_\y U)_{jj} - x_j\big) \exp\Big(-\frac{1}{2t}\text{Tr}(D_\x - U^\dagger D_\y U)^2\Big) \;\mathrm{d}U,
 \end{equation}
  as long as one can justify the differentiation under the integral sign.
  Rewrite the integrand above as
  \begin{equation}
  \frac{1}{t}\Delta(\y)^2 \big((U^\dagger D_\y U)_{jj} - x_j\big) \Big\{\exp\Big(-\frac{1}{4t}\text{Tr}(D_\x - U^\dagger D_\y U)^2\Big) \Big\}^2.
  \label{eq:dQ1Integrand}
  \end{equation}
  For a Hermitian matrix $H$, one can check that $\mathrm{Tr}H^2 = \sum_{i=1}^n h_{ii}^2 + 2\sum_{i<j} |h_{ij}|^2$ and therefore $\mathrm{Tr}(D_\x - U^\dagger D_\y U)^2 = \sum_{i=1}^n \big(x_i - (U^\dagger D_\y U)_{ii}\big)^2 \allowbreak + 2\sum_{i<j} \big| (U^\dagger D_\y U)_{ij}\big|^2$.    
  The quantity 
  \[
  \frac{1}{4t}\big|(U^\dagger D_\y U)_{jj} - x_j\big| \exp\Big(-\frac{1}{4t}\sum_{i=1}^n \big(x_i - (U^\dagger D_\y U)_{ii}\big)^2 \Big) \exp\big(-\frac{1}{2t}\sum_{i<j} \big| (U^\dagger D_\y U)_{ij}\big|^2 \big)
  \]
  is bounded above by a constant and thus \eqref{eq:dQ1Integrand} is bounded above by a constant times $\Delta(\y)^2 \exp\Big(-\frac{1}{4t}\text{Tr}(D_\x - U^\dagger D_\y U)^2\Big)$.
  This combined with \eqref{eq:dQ1} and applying the Harish-Chandra formula again we have
  \begin{align*}
    \left\vert \frac{\partial Q_t}{\partial x_j}(\x,\y) \right\vert
    &\leq Cc_n^\p t^{-n^2/2} \int_{\mathcal{U}(n)} \Delta(\y)^2 \exp\Big(-\frac{1}{4t}\text{Tr}(D_\x - U^\dagger D_\y U)^2\Big) \;\mathrm{d}U \\
    &= C^\p Q_{2t}(\x,\y),
  \end{align*}
  for some constant $C$, $C^\p$ depending only on $n$.
  
  It remains to justify the swapping of the derivative and the integral in \eqref{eq:dQ1}.
  To do this we shall use the following result from \cite[Theorem 16.8]{Bi95}.
 \begin{proposition}\label{prop:DiffUnderIntegral}
  Let $(Y,\mu)$ be a measure space. 
  Suppose that $f(x,y)$ is a continuous and integrable function of $y$ for each $x\in I$, where $I$ can be taken to be $\mathbb{R}$ and that for each $y\in Y$, $\frac{\partial f}{\partial x}(x,y)$ exists. 
  If for each $x^*$ there exists a function $g(x^*,y)$ integrable in $y$ such that $\big|\frac{\partial f}{\partial x}(x,y)\big| \leq g(x^*,y)$ for all $y$ and all $x$ in some neighbourhood of $x^*$, then
  \[
    \frac{\partial}{\partial x}\int_Y f(x,y) \;\mu(\rd{y}) = \int_Y \frac{\partial f}{\partial x}(x,y) \;\mu(\rd{y}).
  \] 
 \end{proposition}
 Since $e^{-\text{Tr}(D_\y-UD_\x U^\dagger)^2/2} \leq 1$ and $\rd{U}$ integrates to one we can apply the above proposition with $g\equiv 1$ which completes the proof. 
\end{proof}

To show \eqref{eq:L2}, we shall use the following result from \cite[Proposition 2.3]{Jo01} which shows that for $1\leq N\leq n$, the $N$-point correlation function of $Q_t$ is given by a determinant:
\[
  \frac{n!}{(n-N)!} \int_{\R^{n-N}} Q_t(\x,\y) \;\rd{y_{N+1}}\ldots\rd{y_n} = \det[\tilde{K}_t(\x,y_i,y_j)]_{1\leq i,j\leq N},
\]
where (see the equation below \cite[equation (2.18)]{Jo01})
\begin{align}
  \tilde{K}_t(\x,u,v) &= -\frac{1}{(2\pi i)^2 t^2} \int_\gamma \rd{z} \int_{\Gamma_L} \rd{w} \; e^{\frac{1}{2t}(w-v)^2-\frac{1}{2t}(z-u)^2} \frac{1}{w-z} \prod_{j=1}^n \frac{w-x_j}{z-x_j} \notag \\
  &\qquad\times \bigg[ (w+z)(w-z) + uz -vw - t\sum_{j=1}^n \frac{x_j(w-z)}{(w-x_j)(z-x_j)} \bigg],
  \label{eq:Npoint2}
\end{align}
where $\gamma$ is a closed contour around the $x_i$'s and $\Gamma_L:t\to L+it$, $t\in\mathbb{R}$ with $L\in\mathbb{R}$ large enough so that $\gamma$ and $\Gamma_L$ do not intersect. 
Then by taking $N=1$, $K_1(\x,y)$ = $\frac{(n-1)!}{n!} \tilde{K}_1(\x,y,y)$.
Observe that the integral formula (\ref{eq:Npoint2}) make clear the symmetry of $\tilde{K}_t$ with respect to the ordering of $x_1,\ldots,x_n$ and that there are no issues if any of the $x_i$'s coincide.

Observe that
\begin{align}\label{eq:L2Split}
 \int_\mathbb{R} K_1(\x,y)^2 \;\mathrm{d}y 
 	&\leq \sup_{y\in\mathbb{R}} K_1(\x,y) \int_\mathbb{R} K_1(\x,y) \;\mathrm{d}y = n! \sup_{y\in\mathbb{R}} K_1(\x,y),
\end{align}
since $\int_{W_n} Q_1(\x,\y) \;\rd{\y} = 1$ for all $x$.
So it suffices to show that $\sup_{\x,y} K_1(\x,y)$ is bounded or equivalently by the translation invariance of $K_1$ which follows from \cite[equation (2.18)]{Jo01}) that $\sup_{\x\in\mathbb{R}^n} K_1(\x,0)$ is bounded.

\begin{lemma}\label{lem:SupK}
 \[
   \sup_{\x\in\mathbb{R}^n} K_1(\x,y) = \sup_{\x\in\mathbb{R}^n} K_1(\x,0) < \infty.
 \]
\end{lemma}

\begin{proof}
  By formula \eqref{eq:Npoint2} we have
  \begin{align}
    nK_1(\x,0) 
    &= -\frac{1}{(2\pi i)^2} \int_\gamma \rd{z} \int_{\Gamma_L} \rd{w} \; e^{-z^2/2}e^{w^2/2} (w+z) \prod_{j=1}^n \frac{w-x_j}{z-x_j} \notag \\
    &\qquad +\frac{1}{(2\pi i)^2} \int_{\gamma} \rd{z} \int_{\Gamma_L} \rd{w} \; e^{-z^2/2} e^{w^2/2} \prod_{j=1}^n \frac{w-x_j}{z-x_j}\sum_{j=1}^n \frac{x_j}{(w-x_j)(z-x_j)} \notag \\
    &=: I_1 + I_2.
    \label{eq:SupK1}
  \end{align}
  Since the integrand is holomorphic we can by Cauchy's theorem deform the contour $\Gamma_L$ so that $L=0$.
  We can also take $\gamma$ to be the closed (rectangular) contour (see Figure \ref{fig:contour}) around $x_1,\ldots,x_n$ composed of four parts $\gamma_t$, $\gamma_b$, $\gamma_r$ and $\gamma_l$, where $\gamma_t : u \to -u + di$, $u\in [-R,R]$, $\gamma_b : u \to u - di$, $u\in [-R,R]$, $\gamma_r : v \to R + vi$, $v\in [-d,d]$, and $\gamma_l : v \to -R - vi$, $v\in [-d,d]$. 
$R := R(x)$ is chosen so that the minimum distance between the contour $\gamma$ and the $x_i$'s is at least $d$.
  We shall consider each parts of the contour separately.
  Denote the contribution from the contour $\gamma_r$ by $I_j(\gamma_r)$, $j=1,2$ and likewise for the others.
  
  \begin{figure}
    \centering
    \begin{tikzpicture}
      \draw[thick, ->] (-6,0) -- (6,0) coordinate (xaxis);
      \draw[thick, ->] (0,-5) -- (0,5) coordinate (yaxis);
    
      \path[draw,blue, line width=0.8pt, postaction=addarrows] (-4,2)
        -- node[midway, xshift=2cm, above, black] {$\gamma_t$} node[midway, xshift=-0.25cm, above, black] {$di$} (4,2)
        -- node[midway, yshift=1cm, right, black] {$\gamma_r$} node[midway, yshift=-0.4cm, right, black] {$R$} (4,-2)
        -- node[midway, xshift=2cm, below, black] {$\gamma_b$} node[midway, xshift=-0.4cm, below, black] {$-di$} (-4,-2)
        -- node[midway, yshift=1cm, left, black] {$\gamma_l$} node[midway, yshift=-0.4cm, left, black] {$-R$} (-4,2);
        
      \path[draw,red, line width=0.8pt, postaction=addarrowsvert] (0,-5) -- (0,5) node[right, yshift=-1cm, black] {$\Gamma_0$};
    
      \draw plot[mark=x, mark options={color=black, scale=1.5}] coordinates { (-3,0) } node[below] {$x_n$};
      \draw plot[mark=x, mark options={color=black, scale=1.5}] coordinates { (-2.0,0) } node[below] {$x_{n-1}$};
      \draw plot[mark=x, mark options={color=black, scale=1.5}] coordinates { (0.5,0) } node[below] {$x_3$};
      \draw plot[mark=x, mark options={color=black, scale=1.5}] coordinates { (1.5,0) } node[below] {$x_2$};
      \draw plot[mark=x, mark options={color=black, scale=1.5}] coordinates { (2.5,0) } node[below] {$x_1$};;
    \end{tikzpicture}
    \caption{Sketch of the contours used in the proof of Lemma \ref{lem:SupK}.} \label{fig:contour}
  \end{figure}

  Since $\vert z-x_j\vert \geq d$ for all $z\in\gamma$ and $j$, we have 
  \[
    \bigg\vert \prod_{j=1}^n \frac{w-x_j}{z-x_j}\bigg\vert 
      = \prod_{j=1}^n \bigg\vert 1 + \frac{w-z}{z-x_j} \bigg\vert
      \leq \bigg(1 + \frac{|w| + |z|}{d}\bigg)^n.
  \]
  On $\gamma_r$, $|z| = |R+vi| = (R^2+v^2)^{1/2} \leq (R^2+d^2)^{1/2}$ and
  \[
    |e^{-z^2/2}| = |e^{-(R^2+2iRv-v^2)/2}| \leq e^{-R^2/2}e^{d^2/2}.
  \]
  Therefore,
  \begin{align*}
    |I_1 & (\gamma_r)| \\
    &\leq \int_{\gamma_r} \frac{|\rd{z}|}{2\pi} \; e^{-R^2/2}e^{d^2/2} \int_\R \frac{\rd{t}}{2\pi} \; e^{-t^2/2} \big(|t| + (R^2+d^2)^{1/2}\big) \bigg(1 + \frac{|t| + (R^2+d^2)^{1/2}}{d} \bigg)^n \\
    &\leq \frac{e^{-R^2/2}e^{d^2/2}}{(2\pi)^2} \text{length}(\gamma_r) f(R),
  \end{align*}
  where $f(R)$ depends on $d$, $n$ and is polynomial in $R$ and hence $I_1(\gamma_r)\to 0$ as $R\to\infty$.

  For $I_2(\gamma_r)$, observe that
  \[
    \bigg|\prod_{j=1}^n \frac{w-x_j}{z-x_j} \sum_{k=1}^n \frac{x_k}{(w-x_k)(z-x_k)} \bigg| \leq \prod_{j=1}^n \bigg|\frac{w-x_j}{z-x_j}\bigg| \sum_{k=1}^n \frac{1}{|z-x_k|} \leq \frac{n}{d} \bigg(1 + \frac{|w| + |z|}{d}\bigg)^n,
  \]
  and so in a similar manner as above we have that
  \[
    |I_2(\gamma_r)| \leq \frac{n}{d}\frac{e^{-R^2/2}e^{d^2/2}}{(2\pi)^2} \text{length}(\gamma_r) g(R),
  \]
  with $g$ a polynomial in $R$.
  Thus, we also have that $I_2(\gamma_r)\to 0$ as $R\to\infty$.
  By symmetry, the same argument shows that $I(\gamma_l)$ also vanishes as $R\to\infty$.
  Thus, we can deform the contour $\gamma$ to the two horizontal lines, $\gamma_+ : u \to -u+di$ and $\gamma_- : u\to u-di$, $u\in\mathbb{R}$.
  On $\gamma_+$, $|z| = (u^2 + d^2)^{1/2}$ and $|e^{-z^2/2}| \leq e^{-u^2/2} e^{d^2/2}$.
  Hence, in the same fashion as above, we have
  \begin{align}
    |I_1(\gamma_+)| 
    &\leq \frac{e^{d^2/2}}{4\pi^2} \int_{\R} \rd{u} \int_\R \rd{t} \; e^{-(u^2 + t^2)/2} (|t| + (u^2 + d^2)^{1/2}) \bigg(1 + \frac{ |t| + (u^2 + d^2)^{1/2}}{d} \bigg)^n \notag \\
    &\leq C_{d,n},
    \label{eq:SupK2}
  \end{align}
  for some constant $C_{d,n}$.
  Similarly, we have
  \begin{align*}
    |I_2(\gamma_+)|
    &\leq \frac{n}{d}\frac{e^{d^2/2}}{4\pi^2} \int_R \rd{u}\int_R \rd{t}\; e^{-(u^2+t^2)/2} \bigg(1 + \frac{ |t| + (u^2 + d^2)^{1/2}}{d} \bigg)^n \\
    &\leq C^\p_{d,n},
  \end{align*}
  for some constant $C^\p_{d,n}$.
  By symmetry, $|I_1(\gamma_-)|$ and $|I_2(\gamma_-)|$ are also bounded by $C_{d,n}$ and $C^\p_{d,n}$ respectively.
  Take $d=1$ and thus we have shown that there exists a constant $C$ independent of $\x$ and depending only on $n$ such that
  \[
    \sup_{\x\in\R^n} K_1(\x,0) \leq  C,
  \]
  which completes the proof.
 \end{proof}

\begin{lemma}
  There exists a constant $C_4>0$ depending only on $n$ such that for all $t>0$ and $\x\in\R^n$,
  \[
    \int_\mathbb{R} K_{t}(\x,y)^2 \;\rd{y} \leq C_4 t^{-1/2}.
  \]
  \label{lem:integralofKSquared}
\end{lemma}

\begin{proof}
 By the scaling property of $Q_t$ and a change of variables 
 \[
  \int_\mathbb{R} K_{t}(\x,y)^2 \;\rd{y} = t^{-1/2} \int_\mathbb{R} K_1(\x t^{-1/2},y^\prime)^2 \;\rd{y^\prime}.
 \]
 By Lemma \ref{lem:SupK} and (\ref{eq:L2Split}), the latter integral for each fixed $n$ is bounded uniformly in $x$ which gives the desired result.
\end{proof}

We are now ready to prove Theorem \ref{thm:DysonKernelCty}. 

\begin{proof}[Proof of Theorem \ref{thm:DysonKernelCty}(a)] Summarizing the argument given above, we take 
\[
R(\x,\z;y_1) := \int_{\R^{n-1}} | Q_1(\x,\y) - Q_1(\z,\y) | \;\rd{\y_*}
\]
and the result follows from Lemma \ref{lem:min} and \eqref{eq:QtScaling2}. The hypothesis of   Lemma \ref{lem:min}  for this choice of $R$  is verified by checking 
  equation \eqref{eq:L2} which in turn  holds by virtue of  Lemma \ref{lem:integralofKSquared}. 
\end{proof}

\begin{proof}[Proof of Theorem \ref{thm:DysonKernelCty}(b)]
  Let $t=u+h$ where $h>0$, then we need to estimate
  \begin{align*}
      \int_0^u \int_\mathbb{R} \bigg( \int_{\R^{n-1}} | Q_{s+h}(\x,\y) - Q_{s}(\x,\y) | \;\rd{\y_*} \bigg)^2 \;\rd{y_1}\rd{s}. 
  \end{align*}
  Making the change of variable $s = hs^\p$, $\y = \sqrt{h}\y^\p$ and using the scaling property (\ref{eq:QtScaling}) of $Q_t$, the above is bounded by
  \[
  h^{1/2} \int_0^\infty \int_\R \bigg( \int_{\R^{n-1}} | Q_{s^\p+1}(\x/\sqrt{h},\y^\p) - Q_{s^\p}(\x/\sqrt{h},\y^\p) | \;\rd{\y_*^\p} \bigg)^2 \;\rd{y_1^\p}\rd{s^\p},
  \]
  and hence it suffices to show that 
  \begin{equation*}
  \int_0^\infty \int_\R \bigg( \int_{\R^{n-1}} | Q_{s+1}(\x,\y) - Q_{s}(\x,\y) | \;\rd{\y_*} \bigg)^2 \;\rd{y_1}\rd{s} < \infty
  \end{equation*}
  uniformly for $x\in\R^n$ which we shall do by dividing the time integral from 0 to 1 and from 1 to $\infty$ and bounding each of them separately.
  Firstly, by Lemma \ref{lem:integralofKSquared}, the non-negativity of $Q_t$ and the inequality $(a + b)^2 \leq 2(a^2 + b^2)$, we have
  \begin{align*}
    \int_0^1 \int_\R \bigg( \int_{\R^{n-1}} |Q_{s+1}(\x,\y) - Q_s(\x,\y)| & \;\rd{\y_*} \bigg)^2 \;\rd{y_1}\rd{s} \\
    &\leq 2 \int_0^1\int_\R K_{s+1}(\x,y)^2 + K_s(\x,y)^2 \;\rd{y}\rd{s} \\
    &< C,
  \end{align*}
  with $C>0$ independent of $\x$.
  On the other hand, we have
  \[
      | Q_{s+1}(\x,\y) - Q_s(\x,\y) | = \bigg| \int_s^{s+1} \frac{\partial Q_r}{\partial r}(\x,\y)  \;\rd{r} \bigg|,
  \]
  and so we need to estimate the derivative of $Q_t$.
  Using the Harish-Chandra formula and denoting $A_U = (D_\y - UD_\x U^\dagger)^2$ we see that
  \begin{equation*}
    \frac{\partial Q_r}{\partial r}(\x,\y) 
    = c_n (2\pi)^{-n/2} r^{-n^2/2} \Delta(\y)^2 \int_{\mathcal{U}(n)} e^{-\mathrm{Tr} A_U/2r} \bigg(\frac{\mathrm{Tr} A_U}{2r^2} - \frac{n^2}{2r}\bigg) \;\rd{U},
  \end{equation*}
  where we have applied Proposition \ref{prop:DiffUnderIntegral} with $g\equiv 1$ to swap the derivative and the integral.
  Then in a similar manner as in the proof of Lemma \ref{prop:L1Derivative} we have
  \begin{align}
    \left|\frac{\partial Q_r}{\partial r}(\x,\y)\right| 
    &\leq \frac{C^\p}{r} c_n (2\pi)^{-n/2} r^{-n^2/2} \Delta(\y)^2 \int_{\mathcal{U}(n)} e^{-\mathrm{Tr} A_U/4r} \;\rd{U} \notag \\
    &= \frac{C}{r} Q_{2r}(\x,\y),
    \label{eq:QtTimeDerivative}
  \end{align}
  for some constant $C>0$ depending only on $n$.
  Therefore, by Minkowski's integral inequality and Lemma \ref{lem:integralofKSquared} 
  \begin{align*}
    \bigg( \int_\R \bigg( \int_{\R^{n-1}} \bigg| \int_s^{s+1} \frac{\partial Q_r}{\partial r}(\x,\y) \;\rd{r} \bigg| & \;\rd{\y_*} \bigg)^2 \rd{y_1}\bigg)^{1/2} \\
    &\leq C\bigg( \int_\R \bigg( \int_{\R^{n-1}} \int_s^{s+1} \frac{1}{r} Q_{2r}(\x,\y) \;\rd{r} \;\rd{\y_*} \bigg)^2 \rd{y_1}\bigg)^{1/2} \\
    &\leq C\int_s^{s+1} \frac{1}{r} \bigg(\int_\R K_{2r}(\x,y_1)^2 \;\rd{y_1}\bigg)^{1/2} \rd{r} \\
    &\leq C^\p s^{-5/4}.
  \end{align*}
  Consequently,
  \[
    \int_1^\infty \int_\R \bigg( \int_{\R^{n-1}} | Q_{s+1}(\x,\y) - Q_s(\x,\y) | \;\rd{\y_*} \bigg)^2 \;\rd{y_1}\rd{s} \leq C^{\p 2} \int_1^\infty s^{-5/2} \;\rd{s} < \infty.
  \]

  Finally, by Lemma \ref{lem:integralofKSquared} we have
  \begin{align*}
    \int_u^t \int_\mathbb{R} K_{s}(\x,y)^2 \;\rd{y} \rd{s} 
    \leq C_4 \int_u^t s^{-1/2} \;\rd{s}
    \leq 2C_4 |t-u|^{1/2}.
  \end{align*}
  This completes the whole proof of the theorem.
\end{proof}

\section{Existence, Uniqueness and Moment Estimates}\label{sec:existence}

\subsection{Bounded Initial Data}

We now prove the existence, uniqueness and moment estimates part of Theorem \ref{thm:MnMain}(a).
The proof of continuity will be delayed to Section \ref{sec:cty}.
In the sequel constants will generally be denoted by $c$, $C$ or $K$ and possibly adorned with primes or subscripts. 
They may differ from line to line and their dependence if any will always be specified.
However, $C_i$, $1\leq i\leq 4$ will always mean the constants in Theorem \ref{thm:DysonKernelCty} and Lemma \ref{lem:integralofKSquared}.
$T>0$ will always denote the finite time horizon.

\begin{proof}[Proof of existence, uniqueness and moment estimates of Theorem \ref{thm:MnMain}(a)]
  The proof is by a Picard iteration argument.
  Throughout the proof, we fix an arbitrary integer $p\geq 2$.
  For $(t,\y)\in(0,\infty)\times\mathbb{R}^n$ define $m^0(t,\y) := J_n(t,\y)$ where $J_n$ was defined in (\ref{eq:MnBounded}) and for $k\geq 1$, let
  \begin{align}
    m^k(t,\y) 
    &= m^0(t,\y) + \frac{1}{(n-1)!} \int_0^t \int_{\mathbb{R}^n} Q_{t-s}(\y,\y^\prime) m^{k-1}(s,\y^\prime) \;\rd{\y_*^\prime} \;\W{s}{y_1^\prime} \notag \\
    &=: m^0(t,\y) + I^k(t,\y).
    \label{eq:PicardSequence}
  \end{align}
  \subsubsection*{Claim 1: The stochastic integrals $I_k$ are well defined}
  We need to show that for all $(t,\y)\in(0,\infty)\times\mathbb{R}^n$, the random field $\big(f_k(s,y), (s,y)\in(0,t)\times\mathbb{R}\big)$ defined by $f_k(s,y_1^\prime) := \int_{\mathbb{R}^{n-1}} Q_{t-s}(\y,\y^\prime) m^k(s,\y^\prime) \;\rd{\y_*^\prime}$ is in $\mathscr{P}_2$ for all $k\geq 0$.
  
  Fix $(t,\y)\in(0,\infty)\times\mathbb{R}^n$ and consider $f_0(s,y_1^\prime) = \int_{\mathbb{R}^{n-1}} Q_{t-s}(\y,\y^\prime) m^0(s,\y^\prime) \;\rd{\y_*^\prime}$.
  We need to show that $m^0$ satisfies the three assumptions of Proposition \ref{prop:predictability2}. 
  Since the initial data $g$ is $\mathscr{F}_0$-measurable, $m^0$ is adapted to the filtration $(\mathscr{F}_t)_{t\geq 0}$.
  By assumption on $g$, $\sup_{\y\in\mathbb{R}^n} \Vert g(\y)\Vert_p \leq K_{p,g} < \infty$ and hence by Minkowski's integral inequality, we have for all $t>0$
  \begin{align}
    \V m^0(t,\y) \V_p  
    &\leq \frac{1}{n!} \int_{\mathbb{R}^n} \V g(\y^\prime) \V_p Q_t(\y,\y^\prime) \;\mathrm{d}\y^\prime \notag \\
    &\leq \bigg(\sup_{\y\in\mathbb{R}^n} \V g(\y) \V_p\bigg) \frac{1}{n!}\int_{\mathbb{R}^n} Q_t(\y,\y^\prime) \;\mathrm{d}\y^\p \notag \\
    &\leq K_{p,g}.
    \label{eq:JngBound}
  \end{align}
  By Lemma \ref{lem:MnJtermCty} below, $(s,\y^\prime) \mapsto m^0(s,\y^\p)$ is $L^2(\Omega)$-continuous over $(0,t)\times\R^n$ and so Proposition \ref{prop:predictability2} implies that $f_0\in\mathscr{P}_2$ and 
  \[
    \int_0^t \int_{\mathbb{R}^n} Q_{t-s}(\y,\y^\prime) m^0(s,\y^\prime) \;\rd{\y_*^\prime} \:\W{s}{y_1^\prime},
  \]
  is a well defined Walsh integral.
  Consequently, the random field $\big(m^1(t,\y) = m^0(t,\y) + I^1(t,\y), (t,\y)\in(0,\infty)\times\mathbb{R}^n\big)$ is well defined.

  We wish to show that the sequence $\{m^k(t,\y)\}_{k\geq 0}$ is Cauchy in $L^p(\Omega)$. 
  To this end, let $d_k(t,\y) := \V m^{k+1}(t,\y) - m^k(t,\y) \V_p$.
  By Lemma \ref{lem:integral}, Lemma \ref{lem:integralofKSquared} and (\ref{eq:JngBound}), we have for all $(t,\y)\in(0,\infty)\times\R^n$,
  \begin{align*}
    d_0(t,\y)^2 
    &\leq A_n^2 c_p^2 \int_0^t\int_\R \bigg(\int_{\R^{n-1}} Q_{t-s}(\y,\y^\p) \V m^0(s,\y^\p) \V_p \;\rd{\y_*^\p} \bigg)^2 \rd{y_1^\p}\rd{s} \\
    &\leq 2K_{p,g}^2 C_4 A_n^2 c_p^2 \sqrt{t} \\
    &= K_{p,g}^2 C_4 A_n^2 c_p^2 \sqrt{\pi} \frac{\sqrt{t}}{\Gamma\big(\frac{3}{2}\big)},
  \end{align*}
  where $\Gamma(3/2) = \sqrt{\pi}/2$ and $A_n$ denotes $1/(n-1)!$.
 
  Now assume for induction that for all $0\leq l\leq k$, $\big(m^l(t,\y), (t,\y)\in(0,\infty)\times\mathbb{R}^n\big)$ is well defined and satisfies
  \begin{itemize}
    \item[(i)] $m^l$ is adapted,
    \item[(ii)] $(s,\y)\mapsto m^l(s,\y)$ is $L^2(\Omega)$-continuous on $(0,t)\times\mathbb{R}^n$ for all $t>0$,
    \item[(iii)] for all $(t,\y)\in(0,\infty)\times\mathbb{R}^n$ and $0\leq l\leq k-1$
      \begin{equation}
        d_l(t,\y)^2 \leq K_{p,g}^2 (C_4A_n^2c_p^2\sqrt{\pi})^{l+1} \frac{t^{(l+1)/2}}{\Gamma\big(\frac{l+1}{2}+1\big)}.  
        \label{eq:induction3}
      \end{equation}
  \end{itemize}
  We want to show that the same is true for $m^{k+1}$ and $d_k$.
  Let $(t,\y)\in(0,\infty)\times\R^n$. 
  Observe that $m^k(t,\y) = m^0(t,\y) + \sum_{l=1}^k m^l(t,\y) - m^{l-1}(t,\y)$, and so to bound the $p$th moments of $m^k$ it suffices to bound each of the $d_l$'s, $0\leq l\leq k-1$. 
  Indeed, by property (iii) and (\ref{eq:JngBound}), we have
  \begin{align}
    \V m^k(t,\y) \V_p^2 
    &\leq 2\V m^0(t,\y) \V_p^2 + \sum_{l=1}^k 2^{l+1} d_{l-1}(t,\y)^2 \notag \\
    &\leq 2K_{p,g}^2 \sum_{l=0}^k (2C_4A_n^2c_p^2\sqrt{\pi})^l \frac{t^{l/2}}{\Gamma\big(\frac{l}{2}+1\big)},
    \label{eq:mkPthMoment}
  \end{align}
  which shows that $\sup_{(s,\y)\in[0,t]\times\mathbb{R}^n} \Vert m^k(s,\y)\Vert_2 < \infty$. 
  This and the induction hypothesis shows that $m^k$ satisfies all three assumptions of Proposition \ref{prop:predictability2} and
  \[
    I^{k+1}(t,\y) = A_n \int_0^t \int_{\mathbb{R}^{n}} Q_{t-s}(\y,\y^\prime) m^k(s,\y^\prime) \;\rd{\y_*^\prime} \:\W{s}{y_1^\prime},
  \]
  is a well defined Walsh integral for all $(t,\y)\in(0,\infty)\times\R^n$.
  Moreover, it is adapted and so $m^{k+1} = m^0 + I^{k+1}$ is also adapted.
  We need to check the $L^2(\Omega)$-continuity of $I^{k+1}$.
  By Theorem \ref{thm:DysonKernelCty} we have for all $0\leq r\leq u\leq t$ and $\y$, $\z\in\mathbb{R}^n$ that
  \begin{align*}
    \Vert I^{k+1}(u,\y) &- I^{k+1}(r,\z) \Vert_2^2 \\
    &\leq 2 A_n^2 \int_0^r\int_\mathbb{R} \bigg(\int_{\mathbb{R}^{n-1}} \big| Q_{u-s}(\y,\y^\prime) - Q_{r-s}(\z,\y^\prime) \big| \Vert m^k(s,\y^\prime)\Vert_2 \;\rd{\y_*^\prime} \bigg)^2 \dd{y_1^\prime}{s} \\
    &\qquad + 2A_n^2 \int_r^u \int_\mathbb{R} \bigg(\int_{\mathbb{R}^{n-1}} Q_{u-s}(\y,\y^\prime) \V m^k(s,\y^\p)\V_2 \;\rd{\y_*^\prime} \bigg)^2 \dd{y_1^\prime}{s} \\
    &\leq 2A_n^2 (C_1+C_2+C_3) \sup_{(s,\y^\prime)\in[0,t]\times\mathbb{R}^n} \Vert m^k(s,\y^\prime) \Vert_2^2 \big(|\y-\z| + |u-r|^{1/2}\big),
  \end{align*}
  which proves the $L^2(\Omega)$-continuity of $m^{k+1}$ on $(0,t)\times\R^n$.
  
  For the bound on $d_k$, we use Lemmata \ref{lem:integral} and \ref{lem:integralofKSquared} and the induction hypothesis to obtain
  \begin{align}
    d_k(t,\y)^2 
    &\leq A_n^2 c_p^2 \int_0^t\int_{\R} \bigg(\int_{\R^{n-1}} Q_{t-s}(\y,\y^\p) d_{k-1}(s,\y^\p) \;\rd{\y_*^\p} \bigg)^2 \rd{y_1^\p} \rd{s} \notag \\
    &\leq K_{p,g}^2 (C_4 A_n^2 c_p^2)^{k+1} \pi^{k/2} \int_0^t \frac{s^{k/2}}{\Gamma\big(\frac{k}{2}+1\big)} (t-s)^{-1/2} \;\rd{s} \notag \\
    &= K_{p,g}^2 (C_4 A_n^2 c_p^2 \sqrt{\pi})^{k+1} \frac{t^{(k+1)/2}}{\Gamma\big(\frac{k+1}{2}+1\big)},
    \label{eq:dkBound}
  \end{align}
where we have used the Euler Beta integral \cite[equation 5.12.1]{OLBC10}: 
  \begin{equation}
    \int_0^1 u^{a-1} (1-u)^{b-1} \;\rd{u} = \frac{\Gamma(a) \Gamma(b)}{\Gamma(a+b)}, \quad a,b>0,
    \label{eq:EulerBeta}
  \end{equation}
  and the fact that $\Gamma(1/2) = \sqrt{\pi}$ to evaluate the time integral.
  It follows that the bound (\ref{eq:mkPthMoment}) holds with $k$ replaced with $k+1$. 
  
  Thus, by induction we conclude that for all integers $k$, the random field $\big(m^k(t,\y) = m^0(t,\y) + I^k(t,\y), (t,\y)\in(0,\infty)\times\mathbb{R}^n\big)$ is well defined and satisfies properties (i), (ii) and (iii) listed above.

  \subsubsection*{Claim 2: The sequence $\{m^k(t,\y)\}_{k\geq 0}$ is Cauchy in $L^p(\Omega)$}
  This follows from the fact that for any $T>0$
  \[
    \sup_{(t,\y)\in[0,T]\times\R^n} \sum_{k=0}^\infty d_k(t,\y) < \infty,
  \]
  which is a consequence of equation \eqref{eq:induction3}, the ratio test and the following asymptotic: $\frac{\Gamma(z+a)}{\Gamma(z+b)} \sim z^{a-b}$, as $z\to\infty$, see \cite[equation 5.11.12]{OLBC10}.
   We conclude that there exist a random field which we denote by $M_n^g(t,\y)$ such that $m^k(t,\y) \to M_n^g(t,\y)$ as $k\to\infty$ in $L^p(\Omega)$ uniformly in $\y\in\mathbb{R}^n$ and $t\in[0,T]$.
  
  \subsubsection*{Claim 3: $M_n^g(t, \y)$ solves equation \eqref{eq:MnBounded}}
  Since each $m^k$ is adapted, $M_n^g$ is also adapted.
  The $L^2(\Omega)$-continuity of $M_n^g$ is inherited from that of $m^k$ since the convergence is uniform on $[0,T]\times\mathbb{R}^n$ for all $T>0$.
  Now take $k\to\infty$ on both sides of (\ref{eq:mkPthMoment}).
  By \cite[Proposition 2.2]{CD13}, we know that for all $x\geq 0$
  \begin{equation}
    e^{x^2}\big( 1 +\mathrm{erf}(x) \big) = \sum_{k=1}^\infty \frac{x^{k-1}}{\Gamma\left( \frac{k+1}{2} \right)}.
    \label{eq:errorFunctionSeries}
  \end{equation}
  Using this with $x = 2C_4A_n^2c_p^2\sqrt{\pi} t^{1/2}$ and the fact that $|\textrm{erf}(\cdot)| \leq 1$ gives the bound (\ref{eq:MnPthMoment}) in the statement of the theorem.
  Thus, by Proposition \ref{prop:predictability2}, for all $(t,\y)\in(0,\infty)\times\R^n$ the random field $f$ defined by $f(s,y_1^\prime) = \int_{\mathbb{R}^{n-1}} Q_{t-s}(\y,\y^\prime) M_n^g(s,\y^\prime) \;\rd{\y_*^\prime}$ for $(s,y_1^\p)\in(0,t)\times\R$ is in $\mathscr{P}_2$ and the stochastic integral
  \[
    I_n(t,\y) = A_n \int_0^t\int_{\mathbb{R}^n} Q_{t-s}(\y,\y^\prime) M_n^g(s,\y^\prime) \;\rd{\y_*^\prime} \:\W{s}{y_1^\prime},
  \]
  is well defined. 

  It remains to show that the limit $M_n^g(t,\y)$ solves (\ref{eq:MnBounded}).
  Fix $(t,\y)\in(0,\infty)\times\mathbb{R}^n$.
  By definition, $m^k(t,\y) = m^0(t,\y) + I^k(t,\y)$ where the left hand side converges in $L^p(\Omega)$ to $M_n^g(t,\y)$.
  For the right hand side we have by the uniform convergence in $L^p(\Omega)$ of $m^k$ that
  \begin{align*}
  \Vert I^k(t,\y) - I_n(t,\y) \Vert_p^2 
  &\leq 2C_4 A_n^2 c_p^2 \sqrt{t} \sup_{(s,\y^\prime) \in [0,t]\times\mathbb{R}^n} \Vert m^{k-1}(s,\y^\prime) - M_n^g(s,\y^\prime)\Vert_p^2 \\
  &\to 0 \quad \text{as } k\to\infty. 
  \end{align*}
  The limit of both sides of the equation $m^k(t,\y) = m^0(t,\y) + I^k(t,\y)$ must be equal almost surely and so we have shown that for each $(t,\y)\in (0,\infty)\times\mathbb{R}^n$, $M_n^g(t,\y)$ satisfies (\ref{eq:MnBounded}) almost surely.
  This proves existence.
 
  \subsubsection*{Claim 4: The solution $M_n^g(t, \y)$ is unique}
  Suppose that $M(t,\y)$ and $N(t,\y)$ are both solutions to (\ref{eq:MnBounded}) with the same initial data $g$ and let $d(t,\y) = \V M(t,\y) - N(t,\y) \V_p$ then by a similar calculation as for existence we have for each $k$,
  \begin{equation}
    d(t,\y)^2 \leq \sup_{(s,\y)\in[0,t]\times\R^n} d(s,\y)^2 \; (C_4A_n^2c_p^2\sqrt{\pi})^k \frac{t^{k/2}}{\Gamma\big(\frac{k}{2}+1\big)},
    \label{eq:uniqueness}
  \end{equation}
  which converges to 0 as $k\to\infty$.
  Therefore, $d\equiv 0$ and so for all $(t,\y)$, $M(t,\y) = N(t,\y)$ almost surely i.e. $M$ and $N$ are versions of each other.
  This proves uniqueness.
\end{proof}

\subsection{Delta Initial Data}

We now turn our attention to the integral equation \eqref{eq:MnDelta}.
In this case, the solutions no longer have $p$th moments which are bounded uniformly in time and so we need a different approach to the one used in the previous subsection. 
Our proof uses the chaos expansion \eqref{eq:MnChaos1} and the local time estimates in Lemmata \ref{lem:localTimeFormula} and \ref{lem:localTimeExpMoments}.

\begin{proof}[Proof of moment estimates and fundamental solution property of  Theorem \ref{thm:MnMain}(b)] 


  \hspace{.1in} \newline Consider $M_n(t,\x,\y)$ defined by the chaos expansion \eqref{eq:MnChaos1}.
  We first estimate its $p$th moments for $p\geq 2$. 
  \subsubsection*{Claim 1: The $p$th moments of $M_n(t,\x,\y)$ satisfies \eqref{eq:MnPthMomentDelta}}
  The approach is to construct an approximating sequence to $M_n$ and estimate the moments of each term of the sequence and take limits.
  The natural candidate for the approximating sequence is the following: for each $(t,\x,\y)\in(0,\infty)\times\mathbb{R}^n\times\mathbb{R}^n$, let $m^0(t,\x,\y) := J_n(t,\x,\y)$ where $J_n$ was defined in (\ref{eq:MnDelta}) and for $k\geq 1$ define
  \[
    m^k(t,\x,\y) = m^0(t,\x,\y)\bigg(1 + \sum_{l=1}^k \int_{\Delta_l(t)} \int_{\mathbb{R}^l} R_l(\mb{s},\mb{y}^\p; t,\x,\y) \;W^{\otimes l}(\db{s},\db{y}^\p) \bigg).
  \]
  In other words, $m^k(t,\x,\y)$ is the $k$th partial sum of the chaos expansion for $M_n(t,\x,\y)$.
  Let $d_k(t,\x,\y) = \V m^{k+1}(t,\x,\y) - m^k(t,\x,\y) \V_p$ then by Lemma \ref{lem:LpChaos}
  \begin{equation}
    d_{k-1}(t,\x,\y)^2 \leq c_p^{2k}m^0(t,\x,\y)^2 \int_{\Delta_k(t)} \int_{\R^k} R_{k}(\mb{s},\mb{y}^\p;t,\x,\y)^2 \;\rd{\y}^\p\rd{\mb{s}}.
    \label{eq:dkPthMoment}
  \end{equation}
  Therefore,
  \begin{align*}
    \Vert m^k(t,\x,\y) \Vert_p^2 
    &\leq 2m^0(t,\x,\y)^2 + \sum_{l=1}^k 2^{l+1} d_{l-1}(t,\x,\y)^2 \\
    &\leq 2m^0(t,\x,\y)^2 \bigg(1 + \sum_{l=1}^k (2c_p^2)^l \int_{\Delta_l(t)} \int_{\mathbb{R}^l} R_l(\mb{s},\mb{y}^\p; t,\x,\y)^2 \;\db{y}^\p\db{s} \bigg).
  \end{align*}
  Each term in the sum above is equal to $(2c_p^2)^l \mathbb{E}_{\x,\y;t}^{X,Y}\big[\big( \sum_{i,j=1}^n L_t(X^i-Y^j) \big)^l \big] / l!$ by Lemma \ref{lem:localTimeFormula} where $X = (X^1,\ldots,X^n)$, $Y = (Y^1,\ldots,Y^n)$ are independent copies of a collection of $n$ non-intersecting Brownian bridges which start at $\x$ in time 0 and end at $\y$ in time $t$.
  Letting $k\to\infty$ we have for all $(t,\x,\y)\in(0,\infty)\times\mathbb{R}^n\times\mathbb{R}^n$
  \begin{equation}
    \lim_{k\to\infty} \Vert m^k(t,\x,\y) \Vert_p^2 \leq 2m^0(t,\x,\y)^2 \mathbb{E}_{\x,\y;t}^{X,Y}\Big[ \exp\Big( 2c_p^2 \sum_{i,j=1}^n L_t(X^i-Y^j) \Big)\Big].
    \label{eq:limMkPthMoment}
  \end{equation}
  For each $t>0$ and $p\geq 2$, Lemma \ref{lem:localTimeExpMoments} shows that the right hand side of the previous display is bounded above by $Ct^{-n^2}$ for all $\x$, $\y\in\R^n$ for some constant $C$ depending on $n$ and $p$.   
  By Cauchy--Schwarz inequality
  \begin{align}
    \V m^k(t,\x,\y) &- m^{k^\p} (t,\x,\y) \V_p^p \notag \\
        &\leq \V m^k(t,\x,\y) - m^{k^\p}(t,\x,\y) \V_2 \V m^k(t,\x,\y) - m^{k^\p}(t,\x,\y) \V_{2(p-1)}^{p-1}.
    \label{eq:mkCauchy}
  \end{align}
  It follows from \cite[Proposition 5.1]{OW11} that \eqref{eq:MnChaos1} is convergent in $L^2(\Omega)$ and so the partial sum $m^k$ converges to $M_n$ in $L^2(\Omega)$. 
  Combining this with the moment bound \eqref{eq:limMkPthMoment} shows that \eqref{eq:mkCauchy} converges to 0 as $k$, $k^\p\to\infty$.
  Therefore, $m^k(t,\x,\y)$ also converges to $M_n(t,\x,\y)$ in $L^p(\Omega)$ and we can replace the left hand side of (\ref{eq:limMkPthMoment}) with $\Vert M_n(t,\x,\y) \Vert_p^2$. 
  
  \subsubsection*{Claim 2: $M_n(t,\x,\y)$ is a solution to \eqref{eq:MnDelta}}
  We show that $M_n(t,\x,\y)$ defined by (\ref{eq:MnChaos1}) satisfies equation (\ref{eq:MnDelta}) for all $(t,\x,\y)\in(0,\infty)\times\R^n\times\R^n$.
  Recall that $M_n(t,\x,\y)$ is well defined on the boundary of the Weyl chamber and it is symmetric under permutations of both its space variables, hence we can extend it to a function on $\R^n\times\R^n$. 
  Similarly we also extend $Q_{t-s}(\x,\y)$ to the whole of $\R^n\times\R^n$. 
  
  First observe that as a consequence of the Markovian structure of   the correlation function $R_k$, see \eqref{eq:Rk},  we have for  $\mb{s}=(s_1,s_2, \ldots, s_k)$ with $s_1<s_2 \ldots<s_k<s<t$  and $\mb{z}=(z_1, z_2,\ldots,z_k) \in \R^k$ together with  $\x,\y,\y^\prime \in \R^n$ that
  \begin{multline}
  \frac{1}{(n-1)!} \int_{\R^{n-1}}Q_{t-s}(\y,\y^\prime) \frac{p_n^*(s,\x,\y^\prime)} {\Delta(\x)\Delta(\y^\prime)} \; R_k(\mb{s},\mb{z}; s,\x,\y^\prime) \;\rd{\y_*^\prime}    \\
 = \frac{p_n^*(t,\x,\y)}{\Delta(\x)\Delta(\y)} R_{k+1}((\mb{s},s),(\mb{z},y_1^\prime); t,\x,\y).
 \label{eq:bridgeMarkovprop}
  \end{multline}
  Thus, substituting the chaos expansion of $M_n$ into the stochastic integral term of \eqref{eq:MnDelta},  the $k$th term of the expansion gives a contribution 
   \begin{multline}
    \label{eq:chaosintegral}
 \frac{1}{(n-1)!} \int_0^t \int_{\R^n}  Q_{t-s}(\y,\y^\prime) \frac{p_n^*(s,\x,\y^\prime)}{\Delta(\x)\Delta(\y^\prime)} \times \\
 \int_{\Delta_k(s)} \int_{\mathbb{R}^k} R_k(\mb{s},\mb{z}; s,\x,\y^\prime)  \mW{k}{\mb{s}}{\z} \:\rd{\y_*^\prime} \:\W{s}{y^\prime_1}     \\
 = \frac{p_n^*(t,\x,\y)}{\Delta(\x)\Delta(\y)} \int_{\Delta_{k+1}(t)} \int_{\mathbb{R}^{k+1}}  R_{k+1}(\mb{s},\mb{z}; t,\x,\y) \mW{k+1}{\mb{s}}{\z}.
  \end{multline}
  The interchange of the order of integration is justified by the stochastic Fubini's theorem (Lemma \ref{lem:stochfubini}, which also holds for a multiple stochastic integral) since for each $s$ and $y_1^\prime$,
  \begin{multline*}
 \int_{\R^{n-1} }  Q_{t-s}(\y,\y^\prime) \left(\frac{p_n^*(s,\x,\y^\prime)} {\Delta(\x)\Delta(\y^\prime)} \right)^2  \int_{\Delta_k(s)} \int_{\mathbb{R}^k} R_k(\mb{s},\mb{z}; s,\x,\y^\prime)^2  \:\rd\mb{s}\: \rd\z\:\rd{\y_*^\prime}  \\
  \leq  \int_{\R^{n-1} }  Q_{t-s}(\y,\y^\prime)  \Vert M_n(s,\x,\y^\prime) \Vert_2^2 \:\rd{\y_*^\prime}<\infty,
  \end{multline*}
  appealing to Lemmata  \ref{lem:localTimeFormula} and  \ref{lem:localTimeExpMoments}.

   Summing over $k$  in  \eqref{eq:chaosintegral}  gives
  \begin{multline}
  \frac{1}{(n-1)!} \int_0^t \int_{\R^n} Q_{t-s}(\y,\y^\prime) M_n(s,\x,\y^\prime) \;\rd{\y_*^\prime} \W{s}{y_1^\prime} \\ 
  = \frac{p_n^*(t,\x,\y)}{\Delta(\x)\Delta(\y)} \sum_{k=1}^\infty \int_{\Delta_k(t)}\int_{\mathbb{R}^k} R_k(\mathbf{s},\mathbf{z}; t,\x,\y) \;\mW{k}{\mb{s}}{\z},
  \label{eq:integralofMn}
  \end{multline}
  provided the sum can be passed though the integrals and the integral on the lefthand side is well defined. 
  To show that the integrand belongs to $\sP_2$ and so the integral is well defined, we first show that the integrand is $L^2(\Omega)$-continuous.
  By \cite[Lemma 6.1]{OW11}, $M_n$ is $L^2(\Omega)$-continuous and the same is true for the above integrand by the same manner as in the proof of Proposition \ref{prop:predictability2} and appealing again to Lemmata \ref{lem:localTimeFormula} and \ref{lem:localTimeExpMoments} to bound the second moment of $M_n$.
  
  Next we calculate, by multiplying the chaos expansions, that
  \begin{multline*}
  \E \bigl[ M_n(s,\x,\y^\prime)M_n(s,\x,\y^{\prime\prime})\bigr]= \frac{p_n^*(s,\x,\y^\prime)}{\Delta(\x)\Delta(\y^\prime)} \frac{p_n^*(s,\x,\y^{\prime\prime})}{\Delta(\x)\Delta(\y^{\prime\prime})} \times \\
   \bigg( 1 + \sum_{k=1}^\infty \int_{\Delta_k(s)}\int_{\mathbb{R}^k} R_k(\mathbf{s},\mathbf{z}; s,\x,\y^\prime) R_k(\mathbf{s},\mathbf{z}; s,\x,\y^{\prime\prime})  \: \rd\mb{s}\:\rd\z  \bigg).
 \end{multline*} 
 Then using this, together with \eqref{eq:bridgeMarkovprop} and monotone convergence,  we deduce that, with $A_n$ denoting $1/(n-1)!$,
 \begin{multline*}
 A_n^2 \int_0^t \int_{\R}  \left\Vert \int_{\R^{n-1}} Q_{t-s}(\y,\y^\prime) M_n(s,\x,\y^\prime) \;\rd{\y^\prime_*} \right\Vert_2^2 \rd{y^\prime_1}\:\rd{s}  \\
 =
 \left( \frac{p_n^*(t,\x,\y)}{\Delta(\x)\Delta(\y)}\right)^2 \sum_{k=1}^\infty \int_{\Delta_k(t)}\int_{\mathbb{R}^k} R_k(\mathbf{s},\mathbf{z}; t,\x,\y)^2 \:\rd\mb{s}\: \rd\z <\infty.
 \end{multline*}
 Therefore, by Proposition \ref{prop:predictability} the integral on the lefthand side of \eqref{eq:integralofMn} is well defined.

 Similarly we can calculate  
 \[
  A_n^2 \int_0^t \int_{\R}  \left\Vert \int_{\R^{n-1}} Q_{t-s}(\y,\y^\prime)\bigl( M_n(s,\x,\y^\prime)-m^k(s,\x,\y^\prime) \bigr) \rd{\y^\prime_*} \right\Vert_2^2 \rd{y^\prime_1}\:\rd{s} ,
  \] 
  where  $m^k(s,\x,\y^\prime)$ is the $k$th partial sum of the chaos expansion for $M_n(s,\x,\y^\prime)$, obtaining,
  \[
   \left( \frac{p_n^*(t,\x,\y)}{\Delta(\x)\Delta(\y)}\right)^2 \sum_{l=k+2}^\infty \int_{\Delta_l(t)}\int_{\mathbb{R}^l} R_l(\mathbf{s},\mathbf{z}; t,\x,\y)^2 \:\rd\mb{s}\: \rd\z.
  \]
    We observe that   this tends to  zero as $k$ tends to infinity, and so prove  the interchange of sum and integral to obtain that  \eqref{eq:integralofMn} was valid.

  Finally it follows  from \eqref{eq:integralofMn} that the right hand side of (\ref{eq:MnDelta})  is equal to
  \[
    \frac{p_n^*(t,\x,\y)}{\Delta(\x)\Delta(\y)} \bigg( 1 + \sum_{k=1}^\infty \int_{\Delta_k(t)}\int_{\mathbb{R}^k} R_k(\mathbf{s},\mathbf{z}; t,\x,\y) \;\mW{k}{\mb{s}}{\z} \bigg), 
  \]
  which is the definition of $M_n(t,\x,\y)$ as required.
 
\subsubsection*{Claim 3: $M_n$ is a fundamental solution to \eqref{eq:MnDelta}}
Let $g$ be a function satisfying the assumptions of part (a) of the theorem and define
 \[
v(t,\y) := \frac{1}{n!} \int_{\R^n} g(\x)M_n(t,\x,\y) \Delta(\x)^2 \;\rd{\x}.
\]
 The integral defining $v$ is well defined; noting that $g(x)$ and $M_n(t,x,y)$ are independent, we have
\[
\E \left [ \int_{\R^n} |g(\x)M_n(t,\x,\y) | \Delta(\x)^2 \;\rd{\x} \right] \leq   \sup_{\x \in \R^n} \Vert g(\x) \Vert _1 \:\int_{\R^n} \Vert M_n(t,\x,\y) \Vert_1 \Delta(\x)^2 \;\rd{\x}< \infty,
\]
 where the integral is finite because
\begin{multline*}
\int_{\R^n} \Vert M_n(t,\x,\y) \Vert_1 \Delta(\x)^2 \;\rd{\x}= \int_{\R^n} \E\bigl[M_n(t,\x,\y)\bigr]\Delta(\x)^2 \;\rd{\x} \\ =  \int_{\R^n} p^*_n(t,\x,\y) \frac{ \Delta(\x)}{\Delta(\y)}\;\rd{x}=n!,
\end{multline*}
noting that $M_n(t,\x,\y)$ is non-negative by \cite[Proposition 5.5]{OW11}.
We claim that
\begin{equation*}
 v(t,\y) 
    = \frac{1}{n!}\int_{\mathbb{R}^n} g(\y^\prime) Q_t(\y,\y^\prime) \;\mathrm{d}\y^\prime + \frac{1}{(n-1)!}\int_0^t \int_{\mathbb{R}^n} Q_{t-s}(\y,\y^\prime) v(s,\y^\prime) \;\rd{\y_*^\prime} \;\W{s}{y_1^\prime}
\end{equation*}
from which it follows, by the  uniqueness of solutions of \eqref{eq:MnBounded}, that $v(t,\y)=M^g_n(t,\y)$ with probability one.

Multiplying both sides of \eqref{eq:MnDelta} by $g(\x)\Delta(\x)^2$ and integrating with respect to $\x$, the claim will follow, provided we justify exchanging the order of  the integrations on the righthand side. Exchanging the integrals with respect to $\x$ and $\y^\prime_*$ is straightforward. To use the stochastic Fubini Theorem to interchange the integral with respect to $\x$ and the stochastic integral it suffices to show that 
\[
 \int_{\mathbb{R}^n} \frac{\Delta(\x)^2}{\rho(\x)}   \int_0^t \int_{\mathbb{R}} \left\Vert \: g(\x) \int_{\mathbb{R}^{n-1}} Q_{t-s}(\y,\y^\prime) M_n(s,\x,\y^\prime) \;\rd{\y_*^\prime}\right\Vert_2^2\:\rd{y_1^\prime}\: \rd{s}\:\rd{\x},
\]
is finite, where $\rho$ is some chosen  positive density  satisfying $ \int_{\R^n} \rho(\x) \Delta(\x)^2 \: \rd{\x} <\infty.$  This follows by taking the measure $\mu$ to be $ \rho(\x) \Delta(\x)^2 \: \rd{\x} $  in Lemma \ref{lem:stochfubini}. Now, equation \eqref{eq:MnDelta} implies that, with $A_n$ denoting $1/(n-1)!$ as previously,
\begin{multline*}
 A_n^2 \int_0^t \int_{\mathbb{R}} \left\Vert \:  \int_{\mathbb{R}^{n-1}} Q_{t-s}(\y,\y^\prime) M_n(s,\x,\y^\prime) \;\rd{\y_*^\prime}\right\Vert_2^2\:\rd{y_1^\prime}\: \rd{s}
+ \left( \frac{p_n^*(t,\x,\y)}{\Delta(\x)\Delta(\y)}\right)^2=\\
\left\Vert M_n(t,\x,\y) \right\Vert_2^2,
\end{multline*}
and consequently, using the independence of $g(\x)$ and  $M_n(t,\x,\y)$ we see that it is enough that
\[
 \Bigl( \sup_\x  \Vert g(\x) \Vert_2^2 \Big) \int_{\mathbb{R}^n} \frac{\Delta(\x)^2}{\rho(\x)}  \Vert M_n(t,\x,\y)\Vert_2^2 \:\rd{\x}
\]
is finite.  In view of the hypothesis on $g$,  and the Gaussian bound on $\Vert M_n(t,\x,\y)\Vert_2^2$ coming from the Harish-Chandra formula and \eqref{eq:supHCIntegrand}, we 
can make  both this quantity and $ \int_{\R^n} \rho(\x) \Delta(\x)^2 \: \rd{\x} $ finite by  choosing, for example,
\[
\rho(\x)= \bigl(1+ \sum x_i^2\bigr)^{-2n^2}.
\]
Thus the application of  the stochastic Fubini Theorem is justified and the result proved.

\end{proof}

\section{Continuity}\label{sec:cty}

We shall use the following version of Kolmogorov's continuity criterion which is due to Chen and Dalang, see \cite[Proposition 4.2]{CD13b}.

\begin{theorem}
  Consider a random field $\big(f(t,\y) : (t,\y)\in\mathbb{R}_+\times\mathbb{R}^d \big)$.
  Suppose there are constants $\alpha_0,\ldots,\alpha_d \in (0,1]$ such that for all $p > 2(d+1)$ and all $M>1$, there exist a constant $C := C(p,M)$ depending on $p$ and $M$ such that 
  \[
    \Vert f(t,\y) - f(s,\x) \Vert_p \leq C \bigg(|t-s|^{\alpha_0} + \sum_{i=1}^d |y_i-x_i|^{\alpha_i} \bigg),
  \]
  for all $(t,\y)$ and $(s,\x)$ in $[1/M,M]\times[-M,M]^d$. 
  Then $f$ has a modification which is locally H\"older continuous on $(0,\infty)\times\mathbb{R}^d$ with indices $(\beta\alpha_0,\ldots,\beta\alpha_d)$ for all $\beta\in(0,1)$.
  \label{thm:Kolmogorov}
\end{theorem}

\subsection{Bounded Initial Data}

We now prove the H\"older continuity of the solution to (\ref{eq:MnBounded}) by verifying the assumptions of Kolmogorov's continuity criterion.
We first estimate the increments of $J_n(t,\y) = \frac{1}{n!} \int_{\R^n} g(\y^\p) Q_t(\y,\y^\p) \;\rd{\y^\p}$ where $g$ satisfies the bound $\sup_{\y\in\R^n} \V g(\y) \V_p \leq K_{p,g}$.

\begin{lemma} 
  Let $M>1$ and $p\geq 2$. 
  There exist constants $K_i := K_i(g,M,n,p)>0$, $i=1,2$ such that for all $t$, $t^\prime \in[1/M,M]$ and $\y$, $\y^\prime\in\R^n$
  \[
    \Vert J_n(t,\y) - J_n(t^\prime,\y) \Vert_p \leq K_1 |t-t^\prime|,
  \]
  and
  \[
    \Vert J_n(t,\y) - J_n(t,\y^\prime) \Vert_p \leq K_2 |\y-\y^\prime|.
  \]
  \label{lem:MnJtermCty}
\end{lemma}

\begin{proof}
  By the assumptions on $g$ and Minkowski's integral inequality, we have
  \begin{align*}
    \Vert J_n(t,\y) - J_n(t^\prime,\y^\p) \Vert_p 
    &\leq \frac{1}{n!} \bigg(\sup_{\z\in\mathbb{R}^n} \Vert g(\z) \Vert_p\bigg) \int_{\mathbb{R}^n} \big|Q_t(\y,\z) - Q_{t^\prime}(\y^\p,\z)\big| \;\rd{\z}.
  \end{align*}
  We first consider the time increment.
  By \eqref{eq:QtTimeDerivative}, there is a constant $C$ depending only on $n$ such that for all $\y\in\R^n$ and $t$, $t^\p\in[1/M,M]$ such that
  \begin{align*}
    \int_{\R^n} |Q_t(\y,\z)-Q_{t^\p}(\y,\z)| \;\rd{\z} 
    &= \int_{\R^n} \bigg| \int_t^{t^\p} \frac{\partial Q_r}{\partial r}(\y,\z) \;\rd{r}\bigg| \;\rd{\z} \\
    &\leq C \int_t^{t^\p} r^{-1} \int_{\R^n} Q_{2r}(\y,\z)\;\rd{\z}\rd{r} \\
    &\leq Cn!M|t^\p-t|.
  \end{align*}
  Similarly for the space increment we need to estimate
  \[
    \int_{\R^n} |Q_t(\y,\z)-Q_t(\y^\p,\z)| \;\rd{\z} = \int_{\R^n} \bigg| \int_0^1 \nabla Q_t\big(\mb{r}(\rho),\z\big)\cdot \mb{r}^\p(\rho) \;\rd{\rho}\bigg| \;\rd{\z},
  \]
  where $\mb{r}(\rho) = (1-\rho)\y + \rho \y^\p$, $\rho\in[0,1]$ is the straight line from $\y^\p$ to $\y$.
  By Lemma \ref{prop:L1Derivative}, $\frac{\partial Q_t}{\partial x_j}(\y,\z) \leq CQ_{2t}(\y,\z)$ and so the previous display is less than
  \begin{align*}
    C \sqrt{n} \int_0^1 |\mb{r}^\p(\rho)| \int_{\R^n} Q_{2t}(\mb{r}(\rho),\z) \;\rd{\z}\:\rd{\rho} \leq C^\p |\y^\p-\y|,
  \end{align*}
  for all $t\in[1/M,M]$ and $\y$, $\y^\p\in\R^n$.
\end{proof}

We now turn our attention to the stochastic integral term $I_n(t,y)$. 

\begin{proposition}
  Let $M>1$ and $p\geq 2$. 
  There exists a constant $K := K(g,M,n,p)$ such that for all $(t,\y)$ and $(u,\z)\in[0,M]\times\mathbb{R}^n$
  \[
    \Vert I_n(t,\y) - I_n(u,\z) \Vert_p \leq K \big( |t-u|^{1/4} + |\y-\z|^{1/2} \big).
  \]
  \label{prop:MnItermCty}
\end{proposition}

\begin{proof}
  We consider the spatial and temporal increment separately.
  By (\ref{eq:MnPthMoment}), there is a constant $C := C(g,M,n,p)$ such that 
  \[
    \sup_{(s,\y^\p)\in[0,M]\times\mathbb{R}^n} \Vert M_n^g(s,\y^\p) \Vert_p^2 \leq C.
  \]  
  Then by Lemma \ref{lem:integral} and Theorem \ref{thm:DysonKernelCty}(a)
  \begin{align*}
    \Vert & I_n(t,\y) - I_n(t,\z) \Vert_p^2 \\
    &\leq C A_n^2 c_p^2 \int_0^t \int_\mathbb{R} \bigg( \int_{\mathbb{R}^{n-1}} | Q_{t-s}(\y,\y^\prime) - Q_{t-s}(\z,\y^\prime) | \;\rd{\y_*^\prime} \bigg)^2 \rd{y_1^\prime} \rd{s} \\
    &\leq C_1 C A_n^2 c_p^2 |\y-\z|,
  \end{align*}
where once again we denote $1/(n-1)!$ by $A_n$.

  For the temporal increment we have (assuming without loss of generality that $0\leq u\leq t\leq M$) that
  \begin{align*}
    \Vert I_n(t,\y) - I_n(u,\y) \Vert_p^2 \leq 2\mathrm{I} + 2\mathrm{II},
  \end{align*}
  where by Theorem \ref{thm:DysonKernelCty}(b)
  \begin{align*}
    \mathrm{I} 
    &:= \bigg\Vert A_n \int_0^u \int_{\mathbb{R}^n} \big| Q_{t-s}(\y,\y^\prime) - Q_{u-s}(\y,\y^\prime) \big| M_n^g(s,\y^\prime) \;\rd{\y_*^\prime} \;\W{s}{y_1^\prime} \bigg\Vert_p^2 \\
    &\leq C A_n^2 c_p^2 \int_0^u \int_\mathbb{R} \bigg( \int_{\mathbb{R}^{n-1}} | Q_{t-s}(\y,\y^\prime) - Q_{u-s}(\y,\y^\prime) | \;\rd{\y_*^\prime} \bigg)^2 \rd{y_1^\prime} \rd{s} \\
    &\leq C_2 C A_n^2 c_p^2 |t-u|^{1/2},
  \end{align*}
  and
  \begin{align*}
    \mathrm{II}
    &:= \bigg\Vert A_n \int_u^t \int_{\mathbb{R}^n} Q_{t-s}(\y,\y^\prime) M_n^g(s,\y^\prime) \;\rd{\y_*^\prime} \;\W{s}{y_1^\prime} \bigg\Vert_p^2 \\
    &\leq C_3 C A_n^2 c_p^2 |t-u|^{1/2}.
  \end{align*}
\end{proof}

By the subadditivity of the function $x\mapsto |x|^\beta$, for $\beta\in(0,1]$ we have
\[
  |\y-\y^\prime|^\beta = \left( \sum_{i=1}^n |y_i-y_i^\prime|^2 \right)^{\beta/2} \leq \sum_{i=1}^n |y_i-y_i^\prime|^\beta.
\]
Lemma \ref{lem:MnJtermCty} and Proposition \ref{prop:MnItermCty} together shows that for all $M>1$ and $p\geq 2$, there is a constant $C := C(g,M,n,p)$ such that for all $(t,\y)$ and $(t^\prime,\y^\prime)$ in $[1/M,M]\times[-M,M]^n$,
\[
  \Vert M_n^g(t,\y) - M_n^g(t^\prime,\y^\prime) \Vert_p \leq C \left( |t-t^\prime|^{1/4} + \sum_{i=1}^n |y_i-y_i^\prime|^{1/2} \right).
\]
Taking $p$ large enough and applying Theorem \ref{thm:Kolmogorov} shows that $M_n^g$ has a version that is locally H\"older continuous on $(0,\infty)\times\mathbb{R}^n$ with indices up to $1/4$ in time and up to $1/2$ in space.

\subsection{Delta Initial Data}

We now turn our attention to $M_n(t,\x,\y)$.
Observe that in this case we cannot apply the method used in Proposition \ref{prop:MnItermCty} directly since the $p$th moments of $M_n(t,\x,\y)$ are not bounded uniformly in time, for instance if $\x=\y$ then 
\begin{align*}
  \V M_n(t,\x,\x) \V_2^2 
  &= \E\big[ J_n(t,\x,\x)^2 + 2J_n(t,\x,\x)I_n(t,\x,\x) + I_n(t,\x,\x)^2 \big] \\
  &= \E\big[ J_n(t,\x,\x)^2 \big] + \E\big[ I_n(t,\x,\x)^2 \big] \\
  &\geq J_n(t,\x,\x)^2,
\end{align*}
where
\[
  J_n(t,\x,\x) = \frac{(2\pi t)^{-n/2}}{\Delta(\x)^2} \Big(1 + \sum_{\substack{\sigma\in S_n \\ \sigma\neq \text{id}}} \text{sgn}(\sigma) \prod_{i=1}^n e^{-(x_i-x_{\sigma(i)})^2/2t} \Big),
\]
which converges to infinity as $t\downarrow 0$.
However, for any $t>0$ fixed we have by (\ref{eq:limMkPthMoment}) and Lemma \ref{lem:localTimeExpMoments} that there is a constant $C := C(n,p,t)$ such that
\begin{align*}
 \Vert M_n(t,\x,\y) \Vert_p^2 
 \leq 2\bigg(\frac{p_n^*(t,\x,\y)}{\Delta(\x)\Delta(\y)}\bigg)^2 \: \mathbb{E}_{\x,\y;t}^{X,Y} \Big[ \exp\Big( 2c_p^2 \sum_{i,j=1}^n L_t(X^i-Y^j) \Big)\Big]
        \leq C t^{-n^2}, 
\end{align*}
uniformly for $\x$, $\y\in\R^n$.
Thus, for all positive times, $M_n$ belongs to the class of initial data in Theorem \ref{thm:MnMain}(a). 
It is clear that at any given time we can restart the equation taking the current solution as the new initial data.
More precisely, let $\tau>0$ and consider the shifted white noise $\dot{W}^\tau(s,y) := \dot{W}(\tau+s,y)$.
Define $M_n^\tau(t,\x,\y) := M_n(\tau+t,\x,\y)$ then it is easy to check by using the semigroup property of $Q_t$ that $M_n^\tau$ satisfies the integral equation
\begin{align*}
  M_n^\tau(t,\x,\y) 
  &= \frac{1}{n!} \int_{\mathbb{R}^n} M_n(\tau,\x,\y^\prime) Q_t(\y,\y^\prime) \;\rd{\y^\p} \\
  &\qquad + \frac{1}{(n-1)!} \int_0^t\int_{\mathbb{R}^n} Q_{t-s}(\y,\y^\prime) M_n^\tau(s,\x,\y^\prime) \;\rd{\y_*^\prime} \;W^\tau(\rd{s},\rd{y_1^\prime}).
\end{align*}
In other words, $M_n^\tau$ is the solution to (\ref{eq:MnBounded}) driven by the shifted noise $\dot{W}^\tau$ with initial data $M_n^\tau(0,\x,\y) = M_n(\tau,\x,\y)$.
Let $M>1$ and $p\geq 2$ then Lemma \ref{lem:MnJtermCty} and Proposition \ref{prop:MnItermCty} applies to show that there is a constant $C := C(M,n,p,\tau)$ such that for all $t$, $t^\p\in[\tau,M]$ and $\y$, $\y^\p\in[-M,M]^n$ and $\x\in\R^n$
\begin{equation}
  \V M_n^\tau(t,\x,\y) - M_n^\tau(t^\p,\x,\y^\p) \V_p \leq C\big(|t-t^\p|^{1/4} + |\y-\y^\p|^{1/2}\big).
  \label{eq:MnForwardCty}
\end{equation}

\subsubsection{Continuity in the Initial Condition}\label{sec:ctyInitialData}

We study the continuity of $\x\mapsto M_n(t,\x,\y)$; in fact we show that $(t,\x,\y)\mapsto M_n(t,x,y)$ is jointly continuous.
Recall the chaos expansion of $M_n(t,\x,\y)$:
\begin{equation}
  M_n(t,\x,\y) = J_n(t,\x,\y) \bigg( 1 + \sum_{k=1}^\infty \int_{\Delta_k(t)} \int_{\mathbb{R}^k} R_k(\mathbf{s}, \mathbf{y}^\prime ; t,\x,\y)  \;\mW{k}{\mathbf{s}}{\mathbf{y}^\prime} \bigg),
  \label{eq:MnChaos}
\end{equation}
where for $0<s_1<\ldots<s_k<t$, $\y_1 = (y_1^1,y_1^2,\ldots,y_1^k)$,
\begin{multline*}
R_k  (\mathbf{s},\mathbf{\y_1} ; t,\x,\y)=  
 ( (n-1)!)^{-k}  \times \\ \int_{(\mathbb{R}^{n-1})^k} \frac{p_n^*(s_1,\x,\y^1) \prod_{i=2}^k p_n^*(s_{i}-s_{i-1},\y^{i-1},\y^{i}) p_n^*(t-s_k,\y^k,\y)}{p_n^*(t,\x,\y)} \;\prod_{i=1}^k\prod_{j=2}^n \rd{y_j^i}.
\end{multline*}
It is easy to see that $J_n(t,\x,\y) = J_n(t,\y,\x)$ and from the expression of $R_k$, one can see that for all $k\geq 1$
\begin{equation}
  R_k(\mb{s},\mb{z}; t,\x,\y) = R_k(t-\tilde{\mb{s}},\tilde{\mb{z}}; t,\y,\x),
  \label{eq:RkSymmetry}
\end{equation}
where $t-\tilde{\mb{s}} := (t-s_k,\ldots,t-s_1)$, $0<t-s_k<\ldots<t-s_1<t$ and $\tilde{\z} := (z_k,z_{k-1},\ldots,z_1)$.
Therefore, it is reasonable to think that each term in the series \eqref{eq:MnChaos} above is symmetric in $\x$ and $\y$ provided one can reverse time in the multiple stochastic integral.
This motivates the following proposition.

\begin{proposition}
  For all $n\geq 1$ and $\y\in\R^n$ the random fields $(M_n(t,\x,\y), (t,\x) \allowbreak \in (0,\infty)\times\mathbb{R}^n)$ and $(M_n(t,\y,\x); (t,\x)\in (0,\infty)\times\mathbb{R}^n)$ are equal in distribution.
  \label{prop:MnSymmetry}
\end{proposition}

\begin{proof}
  Fix $k\geq 1$ and $(t,\x,\y)\in(0,\infty)\times\mathbb{R}^n\times\mathbb{R}^n$.
  Recall the time reversed white noise $\tilde{W}$ defined by $\tilde{W}([0,s]\times A) = \dot{W}([t-s,t]\times A)$ for $s\leq t$ and $A\in\mathscr{B}_b(\R)$.
  Extend $R_k(\mb{s},\mb{z}; t,\x,\y)$ to a function on $L^2([0,t]^k\times\mathbb{R}^k)$ by setting it to be zero for $\mb{s} \notin \Delta_k(t)$.
  Let $\tilde{R}_k$ be the symmetrisation of $R_k$ given by
  \[
    \tilde{R}_k(\mb{s},\mb{z}; t,\x,\y) = \frac{1}{k!} \sum_{\pi\in S_k} R_k(\pi\mb{s},\pi\z; t,\x,\y),
  \]
  where $\pi\mb{s} = (s_{\pi(1)},\ldots,s_{\pi(k)})$ and likewise for $\pi\z$.
  Clearly, we have $\tilde{R}_k(\tilde{\mb{s}},\tilde{\z}; t,\x,\y) = \tilde{R}_k(\mb{s},\mb{z}; t,\x,\y)$.
  Therefore by Lemma \ref{lem:timeReversal} and (\ref{eq:RkSymmetry}), (recall the definition of the multiple stochastic integral in Section \ref{sec:whiteNoise}) 
  \begin{align*}
    \int_{\Delta_k(t)} \int_{\mathbb{R}^k} R_k(\mb{s},\mb{z}; t,\x,\y)  \;\mW{k}{\mb{s}}{\mb{z}} 
    &= \int_{[0,t]^k} \int_{\mathbb{R}^k} \tilde{R}_k(\mb{s},\mb{z}; t,\x,\y) \;\mW{k}{\mb{s}}{\mb{z}} \\
    &= \int_{[0,t]^k} \int_{\mathbb{R}^k} \tilde{R}_k(t-\mb{s},\mb{z}; t,\x,\y) \;\tilde{W}^{\otimes k}(\rd{\mb{s}},\rd{\mb{z}}) \\
    &= \int_{[0,t]^k} \int_{\mathbb{R}^k} \tilde{R}_k(\tilde{\mb{s}},\tilde{\mb{z}}; t,\y,\x) \;\tilde{W}^{\otimes k}(\rd{\mb{s}},\rd{\mb{z}}) \\
    &=\int_{\Delta_k(t)} \int_{\mathbb{R}^k} R_k(\mb{s},\mb{z}; t,\y,\x) \;\tilde{W}^{\otimes k}(\rd{\mb{s}},\rd{\mb{z}}).
  \end{align*}
  Thus, applying the above to each term of the sum in (\ref{eq:MnChaos}) we see that 
  \begin{align*}
    M_n(t,\x,\y) 
    &=J_n(t,\y,\x)\bigg(1 + \sum_{k=1}^\infty \int_{\Delta_k(t)} \int_{\R^k} R_k(\mb{s},\mb{z};t,\y,\x) \;\tilde{W}^{\otimes k}(\db{s},\db{z}) \bigg) \\
    &\stackrel{(d)}{=}  M_n(t,\y,\x),
  \end{align*}
  for all $(t,\x,\y)\in(0,\infty)\times\R^n\times\R^n$ and the result follows.
 \end{proof}

Finally, we return to proving the joint continuity of the solution to (\ref{eq:MnDelta}).
We bound $\Vert M_n(t,\x,\y) - M_n(t^\prime,\x^\prime,\y^\prime) \Vert_p^2$ by considering the increments in each variables separately.
Since $M_n(t,\x,\y) = M_n^\tau(t-\tau,\x,\y)$ for $t\geq 2\tau$, we have by Proposition \ref{prop:MnSymmetry} and (\ref{eq:MnForwardCty}) that for all $M>1$ and $p\geq 2$ there is a constant $C := C(M,n,p,\tau)$ such that for all $(t,\x,\y)$ and $(t^\p,\x^\p,\y^\p) \in [2\tau,M]\times[-M,M]^n\times[-M,M]^n$
\begin{align*}
  \Vert & M_n(t,\x,\y) - M_n(t^\p,\x^\p,\y^\p) \Vert_p \\
  &\quad\leq \V M_n^\tau(t-\tau,\x,\y) - M_n^\tau(t^\p-\tau,\x,\y^\p) \V_p + \V M_n^\tau(t^\p-\tau,\y^\p,\x) - M_n^\tau(t^\p-\tau,\y^\p,\x^\p) \V_p \\
  &\quad\leq C\big( |t-t^\p|^{1/4} + |\x-\x^\p|^{1/2} + |\y-\y^\p|^{1/2} \big).
\end{align*}
Since $\tau>0$ is arbitrary, we can take $2\tau = 1/M$ and thus we have shown that there exists a constant $\tilde{C} = \tilde{C}(M,n,p)$ such that for all $(t,\x,\y)$ and $(t^\prime,\x^\prime,\y^\prime) \in [1/M,M]\times[-M,M]^{2n}$ the above inequality holds with $\tilde{C}$ in place of $C$.
Finally, using the subadditivity of $\x\mapsto |\x|^\beta$ for $\beta\in(0,1]$ and applying Theorem \ref{thm:Kolmogorov} proves the existence of a H\"older continuous version.
This concludes the entire proof of Theorem \ref{thm:MnMain}.

\section{Strict Positivity}\label{sec:positivity}

\subsection{A Weak Comparision Principle}

Recall that $K_n(t,\x,\y)$ can be expressed as $K_n(t,\x,\y) = \det[u(t,x_i,y_j)]_{i,j=1}^n$ where $u(t,x,y)$ is the solution to (\ref{eq:SHEDeltaX}) with initial data $\delta_x$.
Bertini--Cancrini \cite{BC95} proved that $u(t,x,y)$ is the limit in $L^p(\Omega)$ for all $p\geq 2$ of $u^\varepsilon(t,x,y)$ as $\varepsilon\to\infty$, where $u^\varepsilon(t,x,y)$ is the solution to the stochastic heat equation subject to a mollified white noise $W^\varepsilon$ in place of the space-time white noise.
Its solution is given by the following Feymann--Kac formula which is well defined for the noise $W^\varepsilon$:
\[
  u^\varepsilon(t,x,y) = p_t(x-y) \mathbb{E}_{x,y;t}^b \bigg[\mathcal{E}\mathrm{xp}\Big(\int_0^t W^\varepsilon(s,b_s) \;\rd{s} \Big)\bigg],
\]
where the expectation is with respect to a Brownian bridge $b$ starting from $x$ at time 0 and ending in $y$ at time $t$.
By the above Feymann--Kac formula it is then clear that for all $(t,x,y)\in(0,\infty)\times\mathbb{R}\times\mathbb{R}$, with probability 1, $u(t,x,y)\geq 0$.
Using this and the determinant formula for $K_n$, the authors in \cite[Proposition 5.5]{OW11} proved by a path switching argument that $K_n(t,\x,\y) \geq 0$ almost surely, for all $(t,\x,\y)\in (0,\infty)\times W_n\times W_n$.

In fact, a stronger result is true since the above implies that $K_n(t,\x,\y) \geq 0$ for all rational points $(t,\x,\y)$ almost surely.
It is well known that $(t,x,y)\mapsto u(t,x,y)$ has a jointly continuous version and hence the same is true for $K_n$ as it is just a sum of products of the $u$'s.
Therefore, by continuity
\begin{equation*}
  \mathbb{P}[ K_n(t,\x,\y) \geq 0 \text{ for all } t>0 \text{ and } \x,\y\in W_n] = 1.
  \label{eq:KnDeltaPositive}
\end{equation*}
Since the Vandermonde determinant $\Delta(\x) = \prod_{1\leq i<j\leq n} (x_i-x_j)$ is non-negative for $\x\in W_n$, we see that $M_n(t,\x,\y) = K_n(t,\x,\y)/\big(\Delta(\x)\Delta(\y)\big) \geq 0$ for all $x$, $y\in W_n^\circ$ almost surely.
By the continuity of $M_n$ proved in the previous section, this non-negativity extends to the boundary of the Weyl chamber and by symmetry to the whole of $\R^n$.
That is,
\begin{equation}
  \mathbb{P}[ M_n(t,\x,\y) \geq 0 \text{ for all } t>0 \text{ and } \x,\y\in \R^n] = 1.
  \label{eq:MnDeltaPositive}
\end{equation}
The next lemma extends this to solutions $M_n^g(t,\y)$ of equation (\ref{eq:MnBounded}) with non-negative initial data $g$ and in fact by the linearity of the equation this is equivalent to a weak comparison principle.

\begin{lemma}[Weak comparison principle]
  Let $M_n^1(t,\y)$ and $M_n^2(t,\y)$ be the solution to \eqref{eq:MnBounded} with initial data $g_1$ and $g_2$ respectively.
  Assume in addition to the assumptions of Theorem \ref{thm:MnMain}(a) that $g_1\geq g_2$, then 
  \[
    \mathbb{P}[M_n^1(t,\y) \geq M_n^2(t,\y) \text{ for all } t>0 \text{ and } \y\in\mathbb{R}^n] = 1.
  \]
  \label{lem:weakComparison}
\end{lemma}

\begin{proof}
  By linearity of the equation \eqref{eq:MnBounded}, it suffices to prove the lemma in the case $g_1 = g$ and $g_2 = 0$.
  By Theorem \ref{thm:MnMain}(b) we have that
  \[
    M_n^g(t,\y) = \frac{1}{n!} \int_{\R^n} g(\x)M_n(t,\x,\y) \Delta(\x)^2 \;\rd{\x}.
  \] 
  From \eqref{eq:MnDeltaPositive} and the non-negativity of $g$ we see that for all $(t,\y) \in[0,\infty)\times\R^n$, $M_n^g(t,\y) \geq 0$ almost surely which combined with the continuity of $(t,\y)\mapsto M_n^g(t,\y)$ gives the conclusion of the lemma.
\end{proof}

\subsection{A Strong Comparison Principle}

We now prove a strong comparision principle of which Theorem \ref{thm:strictPositivity} is an easy corollary.
\begin{theorem}[Strong comparision principle]
  \hfill
  \begin{itemize}
    \item[(a)] Let $M_n^{1}(t,\y)$ and $M_n^{2}(t,\y)$ be two solutions to (\ref{eq:MnBounded}) with initial data $g_1$ and $g_2$ respectively where $g_1$ and $g_2$ are as in Theorem \ref{thm:MnMain}(a) and are also continuous.
  If furthermore $g_1\geq g_2$ and $g_1(\y) > g_2(\y)$ for some $\y\in\mathbb{R}^n$ almost surely, then
  \[
    \mathbb{P}[ M_n^{1}(t,\y) > M_n^{2}(t,\y) \text{ for all } t>0 \text{ and } \y\in\mathbb{R}^n ] = 1.
  \]
\item[(b)] Let $M_n(t,\x,\y)$ be a solution to (\ref{eq:MnDelta}), then
  \[
    \mathbb{P}[ M_n(t,\x,\y) > 0 \text{ for all } t>0 \text{ and } \x,\y\in\R^n] = 1.
  \]
  \end{itemize}
  \label{thm:strongComparision}
\end{theorem}

We begin with a lemma which provides a lower bound for the deterministic term $J_n(t,\y)$ in (\ref{eq:MnBounded}). Recall that Dyson Brownian motion describes the eigenvalues of a Hermitian Brownian motion; the eigenvalues of a GUE matrix  have  the law of the Dyson Brownian motion at time $1$ when it is started from the origin, given by \eqref{eq:QtBoundary} with $a=0$ and $t=1$. 

\begin{lemma}
    Let $\beta := \frac{1}{2} \min_{k=0,1,\ldots,n} \mathbb{P}_{\mathrm{GUE}}[\lambda_{min}^{(n-k)} \geq 1] \mathbb{P}_{\mathrm{GUE}}[\lambda_{max}^{(k)} \leq -1]$ with the interpretation that if $k$ or $n-k$ is equal to 0 then the corresponding probability is equal to 1 and where $\lambda_{min}^{(k)}$, $\lambda_{max}^{(k)}$ denotes the smallest and largest eigenvalue of a $k\times k$ GUE matrix respectively.
  For all $h>h_0>0$, $t>0$, $M>0$, there exists an $m_0 := m_0(h_0,M,n,t)$ such that for all $m \geq m_0$, all $s \in [t/2m, t/m]$ and $\x\in W_n$,
\[
  Q(h,s,\x) := \int_{W_n} Q_s(\x,\y) 1_{(-h,h)^n}(\y) \;\rd{\y} \geq \beta 1_{(-h-M/m, h+M/m)^n}(\x).
\]
  \label{lem:positiveLemma1}
\end{lemma}

\begin{proof}
Fix $h>0$ and an enlargement parameter $\sigma$. 
We first bound $Q(h,s,\x)$ from below for small times $s$ and $\x = (x_1,\ldots,x_n)$ satisfying $h + \sigma \sqrt{ s } \geq x_1 \geq \ldots \geq x_n \geq -h - \sigma \sqrt{s}$.
Set $\delta= 2h/(n+2)$, and consider $K\geq 0$ to be chosen later but satisfying $ 2 K\sqrt{s} \leq  \delta$.
Certainly,
\begin{equation*}
  Q(h, s, \x)\geq  \int_{ \substack{y_1> \ldots >y_n \\ \max_i |y_i-x_i| \leq K\sqrt{s}} } Q_s(\x,\y) 1_{(-h,h)^n}(\y) \;\rd{\y},
\end{equation*}
Now for $k \in \{1, \ldots, n-1\}$ and indeed trivially for $k=0$ or $k=n$ also, we may write by means of the Laplace expansion of determinants 
\begin{equation}
  Q_s(\x,\y) = \sum_{\substack{S\subseteq \{1,2,\ldots,n\} \\ |S|=k}}  \text{sgn}(S) \frac{\Delta_{S}(\y)}{\Delta_{[1,k]}(\x)}
  Q_s\bigl(\x^{[1,k]}, \y^S\bigr)
  Q_s\bigl(\x^{[k+1,n]}, \y^{S^c}\bigr),
  \label{eq:Laplace}
\end{equation}
where
\begin{equation*}
\frac{\Delta_{S}(\y)}{\Delta_{[1,k]}(\x)}= \frac{\prod_{i\in S, j\in S^c} (y_i-y_j)}{ \prod_{i\in [1,k], j\in [k+1,n]}(x_i-x_j)},
\end{equation*}
$\textrm{sgn}(S) = (-1)^{(\sum_{i\in S} i + \sum_{j\in[1,k]} j )}$, $\y^S = (y_1^S,\ldots,y_k^S) = (y_i)_{i\in S}$ and similarly for $\x^{[1,k]}$, $\x^{[k,n+1]}$ and $\y^{S^c}$ where $[1,k] = \{1,2,\ldots,k\}$ and $S^c$ is the complement of $S$.

If we set $x_0 = h + \sigma\sqrt{s}$ and $x_{n+1}=  -h - \sigma \sqrt{s}$ then there must exist a $k\in\{0, 1, \ldots,n  \}$ with $x_{k}-x_{k+1} >\delta$.
We first consider the case when $k\notin\{0, n\}$.
With such a choice of $k$, we obtain from the Laplace expansion,
\begin{equation*}
  Q(h, s, \x) \geq  Q(h,s,K,\x) +\sum_{\substack{S\subseteq \{1,2,\ldots,n\} \\ |S|=k, S \neq [1,k]}} 
R(h,s,K,\x,S)
\end{equation*}
where 
\begin{multline*}
  Q(h,s,K,\x) = \\
  \int_{ \substack{y_1> \ldots >y_n \\ \max_i |y_i-x_i| \leq K\sqrt{s}} }  \frac{\Delta_{[1,k]}(\y)}{\Delta_{[1,k]}(\x)} 
  Q_s\bigl( \x^{[1,k]}, \y^{[1,k]}\bigr) Q_s\bigl( \x^{[k+1,n]}, \y^{[k+1,n]}\bigr)
  1_{(-h,h)^n}(\y) \;\rd{\y}
\end{multline*}
which corresponds to the term in the sum \eqref{eq:Laplace} with $S = \{1,2,\ldots,k\}$ and $R(h,s,K,\x,S)$ are the analogous terms from taking the other possibilities for $S$.

Let us bound each $|R(h,s,K,\x,S)|$ from above. 
We have, assuming that $\sigma \sqrt{s} \leq \delta$, 
\begin{equation*}
  \left|\frac{\prod_{i\in S, j\in S^c} (y_i-y_j)}{ \prod_{i\in [1,k], j\in [k+1,n]}(x_i-x_j)} \right| \leq \left( \frac{2(h+\sigma\sqrt{s})}{\delta} \right)^{k(n-k)} \leq (n+4)^{k(n-k)}=: C(k,n).
\end{equation*}
Enlarging the domain of integration gives
\begin{align*}
  |R(h,s,K,\x, & S)| \\
  &\leq C(k,n) \int_{ \substack{y_1> \ldots >y_n \\ \max_i |y_i-x_i| \leq K\sqrt{s}} }  Q_s\bigl( \x^{[1,k]}, \y^{S}\bigr) Q_s\bigl( \x^{[k+1,n]}, \y^{S^c}\bigr)
 1_{(-h,h)^n} (\y) \;\rd{\y} \\
  &\leq C(k,n) \int_{ \substack{y^S_1> \ldots >y^S_k \\y^{S^c}_1> \ldots >y^{S^c}_{n-k}  \\ \max_i |y_i-x_i| \leq K\sqrt{s}} }  
  Q_s\bigl( \x^{[1,k]}, \y^{S}\bigr) Q_s\bigl( \x^{[k+1,n]}, \y^{S^c}\bigr) \;\rd{\y} \\
  &= C(k,n) \int_{ \substack{y^S_1> \ldots >y^S_k   \\ \max_i |y^S_i-x^S_i| \leq K\sqrt{s}} }
  Q_s\bigl( \x^{[1,k]}, \y^{S}\bigr) \;\rd{\y^S} \\
  &\qquad\qquad\times \int_{ \substack{y^{S^c}_1> \ldots >y^{S^c}_{n-k}  \\ \max_i |y^{S^c}_i-x^{S^c}_i| \leq K\sqrt{s}}  }
  Q_s\bigl( \x^{[k+1,n]}, \y^{S^c}\bigr) \;\rd{\y^{S^c}} \\
  &\leq C(k,n) \int_{ \substack{y^S_1> \ldots >y^S_k   \\ \max_i |y^S_i-x^S_i| \leq K\sqrt{s}} }
  Q_s\bigl( \x^{[1,k]}, \y^S\bigr) \;\rd{\y^S}.
\end{align*}
Denote the elements of $S$ as $S(1) > S(2) > \cdots > S(k)$ and the $i$th component of $\x^S$ as $x_{S(i)} = x_i^S$. 
Since $S \neq [1,k]$ there must exists at least one $j \in \{1,\ldots, k\}$ such that $S(j) \in [k+1,n]$.
Suppose that $\z\in\R^k$ satisfies $ |z_i-x^S_i| \leq K\sqrt{s}$ for all $i=1,\ldots k$ and recall that $2K\sqrt{s} \leq \delta$. Then we can estimate
\begin{equation*}
|z_j -x_j| \geq |x^S_j- x_j| -  |z_j -x^S_j|\geq \delta-  K\sqrt{s}  \geq  K\sqrt{s}.
\end{equation*}
Consequently, the last displayed integral above is bounded above by 
\begin{equation*}
 C(k,n) \int_{\substack{ z_1> \ldots >z_k \\ |z_j -x_j| \geq K\sqrt{s}} } Q_s\bigl( \x^{[1,k]}, \z\bigr) \;\rd{\z} 
\end{equation*}
Now $Q_s\bigl( \x^{[1,k]}, \cdot \bigr)$ is the density of the ordered eigenvalues of the matrix $M=D+\sqrt{s}G$ where $D$ diagonal matrix with entries ${\x}^{[1,k]}$ and $G$ is a $k\times k$ GUE matrix.
By Weyl's eigenvalue inequality \cite[Theorem III.2.1]{Bh97} the $i$th eigenvalue of $M$ lies in the interval $[x_i+\sqrt{s}\lambda_{min}^{(k)}, x_i+ \sqrt{s}\lambda_{max}^{(k)}]$ where $\lambda_{min}^{(k)}$ and $\lambda_{max}^{(k)}$ denote the smallest and the largest eigenvalue of $G$ respectively. 
Thus, we obtain
\begin{equation*}
  |R(h,s,K,\x,S)| \leq  C(k,n) \mathbb{P}\bigl[ \min(\lambda_{min}^{(k)}, -\lambda_{max}^{(k)} ) \geq K\bigr].
\end{equation*}

We turn to bounding $Q(h,s,K,\x)$ from below.  
For $\y$ in the region of integration we have
\begin{align*}
  \frac{\prod_{i\in [1,k], j\in [k+1,n]} (y_i-y_j)}{ \prod_{i\in [1,k], j\in [k+1,n]}(x_i-x_j)}  
  &\geq \prod_{i\in [1,k], j\in [k+1,n]}\left( \frac{x_i-x_j-2K\sqrt{s})}{x_i-x_j} \right) \\
  &\geq \left(1- \frac{2K\sqrt{s}}{\delta}\right)^{k(n-k)} \\
  &=: C(K,s,\delta,k,n).
\end{align*}
Then
\begin{align*}
  C(&K,s,\delta,k,n)^{-1} Q(h,s,K,\x) \\
  &\geq \int_{ \substack{y_1> \ldots >y_n \\ \max_i |y_i-x_i| \leq K\sqrt{s}} }  
  Q_s\bigl( \x^{[1,k]},\ \y^{[1,k]}\bigr) Q_s\bigl( \x^{[k+1,n]}, \y^{[k+1,n]}\bigr)
  1_{(-h,h)^n}(\y) \;\rd{\y} \\
  &= \int_{ \substack{y_1> \ldots >y_k   \\ \max_{i\in[1,k]} |y_i-x_i| \leq K\sqrt{s}} }
  Q_s\bigl( \x^{[1,k]},\y^{[1,k]}\bigr) 1_{(-h,h)^k} (\y^{[1,k]}) \;\rd{\y^{[1,k]}} \\
  &\qquad\times \int_{ \substack{y^{}_{k+1}> \ldots >y_{n}  \\ \max_{i\in[k+1,n]} |y_i-x_i| \leq K\sqrt{s}} }
  Q_s\bigl( \x^{[k+1,n]}, \y^{[k+1,n]}\bigr) 1_{(-h,h)^{n-k}} (\y^{[k+1,n]}) \;\rd{\y^{[k+1,n]}}
\end{align*}
where the integral factorises by virtue of the inequality $2K\sqrt{s} \leq \delta$. 
We now bound each of these two resulting integrals. 
Assume that $2\sigma \sqrt{s} \leq \delta$.  
Considering the first factor, and applying Weyl's inequality to the matrix $M$ as before, noting that $x_1-\sigma \sqrt{s} \leq h$ and $x_k - K\sqrt{s}\geq -h$, gives a lower bound of 
\begin{equation}
  \mathbb{P}\left[  \lambda_{max}^{(k)} \leq -\sigma \text{ and } \lambda_{min}^{(k)} \geq -K \right].
  \label{eq:lowerBound1}
\end{equation}
Similarly, noting that $x_{k+1} + K\sqrt{s} \leq h$ and $x_n + \sigma\sqrt{s} \geq -h$, the second factor is bounded below by 
\begin{equation}
  \mathbb{P}\left[ \lambda_{max}^{(n-k)} \leq K \text{ and } \lambda_{min}^{(n-k)} \geq \sigma \right].
  \label{eq:lowerBound2}
\end{equation}
Thus, we obtain the following lower bound,
\begin{multline*}
  Q(h,s,K,\x) \geq 
  C(K,s,\delta,k,n) \mathbb{P}\left[  \lambda_{max}^{(k)}  \leq -\sigma \text{ and } \lambda_{min}^{(k)} \geq -K \right] \\
  \times \mathbb{P}\left[ \lambda_{max}^{(n-k)}  \leq K\text{ and } \lambda_{min}^{(n-k)} \geq \sigma \right].
\end{multline*}

We can make $\mathbb{P}\left[ \lambda_{max}^{(n-k)} \geq K\right]$ and $\mathbb{P}\left[  \lambda_{min}^{(k)} \leq- K\right]$ arbitrarily small, simultaneously for all $k=1,2, \ldots n$, by choosing $K$ large enough. 
Then, for a chosen $K$, $C(K,s,\delta,k,n)$ can be made arbitrarily close to $1$ by choosing any sufficiently small $s$, and moreover the desired inequality $2 K\sqrt{s} \leq \delta$, will also hold for all sufficiently small $s$ too.  

In the case $k=0$, there is only one term in \eqref{eq:Laplace} and 
\begin{align*}
  Q(h,s,K,\x) 
  &= \int_{ \substack{y_{1}> \ldots > y_{n} \\ \max_{i\in[1,n]} |y_i-x_i| \leq K\sqrt{s}} }
  Q_s\bigl( \x^{[1,n]}, \y^{[1,n]}\bigr) 1_{(-h,h)^{n}} (\y^{[1,n]}) \;\rd{\y^{[1,n]}}.
\end{align*}
By the same reasoning as above and noting that in this case $x_1 - x_0 > \delta$ so that $x_1 + K\sqrt{s} \leq h$ and $x_n + \sigma\sqrt{s} \geq -h$, we have the lower bound \eqref{eq:lowerBound2}.
The same argument applies in the case $k=n$ to obtain the lower bound \eqref{eq:lowerBound1}.
Thus we have shown that given any $\varepsilon>0$ there exists a $s_0>0$ depending on $\varepsilon$, $n$, $h$ and $\sigma$ alone so that
\begin{equation*}
  Q(h,s,\x) + \varepsilon  \geq \beta(n,\sigma)>0,
\end{equation*}
for all $s<s_0$ and $\x$ satisfying $h + \sigma\sqrt{s} \geq x_1 \geq \ldots \geq x_n \geq -h - \sigma \sqrt{s}$, where 
\begin{equation*}
  \beta(n,\sigma) = \min_{k=0,1,\ldots n} \mathbb{P}\left[ \lambda_{min}^{(n-k)} \geq \sigma \right]  {\mathbb P}\left[ \lambda_{max}^{(k)}  \leq -\sigma \right]  
\end{equation*}
with the interpretation that one of the probability on the right hand side is equal to 1 if $k$ or $n-k$ is zero. 

We now deduce the statement of the lemma from the above.
Set $\sigma = 1$, $\varepsilon = \beta(n,1)/2$ and fix $h_0>0$ then there is a $s_0 = s_0(h_0,n)$ such that for all $s<s_0$ and $\x\in (-h_0-\sqrt{s}, h_0+\sqrt{s})^n$ we have
\begin{equation}
  Q(h_0,s,\x) \geq \frac{1}{2}\beta(n,1).
  \label{eq:QLowerBound}
\end{equation}
Now fix $M>0$, $t>0$ then for all $m\geq 2M^2/t$ and $s\in[t/2m,t/m]$ we have $\sqrt{s} \geq M/m$.
Therefore, the infimum of $Q(h_0,s,\x)$ over $x\in(-h_0-M/m, h_0+M/m)^n$ is larger than the infimum of the same quantity over $\x\in(-h_0-\sqrt{s}, h_0+\sqrt{s})^n$.
Hence, choosing $m_0$ such that $m_0 \geq 2M^2/t$ and $t/m_0 \leq s_0$ we have that the inequality \eqref{eq:QLowerBound} holds for all $m\geq m_0$, $s\in [t/2m,t/m]$ and $\x\in(-h_0-M/m,h_0+M/m)^n$.
Moreover, if $h > h_0$ then $\delta = 2h/(n+2) > \delta_0 = 2h_0/(n+2)$ and thus with the same $s_0$ and hence the same $m_0 = m_0(h_0,M,n,t)$ as above, the inequality \eqref{eq:QLowerBound} holds with $h$ in place of $h_0$ for all $m\geq m_0$, $s\in[t/2m,t/m]$ and $\x\in (-h-M/m, h+M/m)^n$.
\end{proof}

\begin{lemma}
  Let $\beta$ be the constant in Lemma \ref{lem:positiveLemma1}.
  Let $t>0$, $M>0$ and $h>h_0>0$ be such that $(-h,h)\subseteq (-2M,2M)$ and let $M_n^g$ be the solution to (\ref{eq:MnBounded}) with initial data $g = 1_{(-h,h)^n}$.
  Then, there exists an $m_0 := m_0(h_0,M,n,t)$ such that for all $m\geq m_0$ 
  \[
    \mathbb{P}\Big[ M_n^g(s,\y) \geq \frac{\beta}{2} 1_{(-h-M/m,h+M/m)^n}(\y) \text{ for all } s\in[t/2m, t/m] \text{ and } \y\in\mathbb{R}^n \Big] \geq 1 - \delta(m),
  \]
  where $\delta(m)$ is such that $(1-\delta(m))^m \to 1$ as $m\to\infty$.
  \label{lem:positiveInductionStep}
\end{lemma}

\begin{proof}
  Let $\beta$ be as in Lemma \ref{lem:positiveLemma1} and let $M>0$, $t>0$, $h>h_0>0$ be given, then by Lemma \ref{lem:positiveLemma1} there exist an $m_0 = m_0(h_0,M,n,t)$ such that for all $m\geq m_0$, all $s\in [t/2m,t/m]$ and $\y\in\mathbb{R}^n$
  \[
    J_n(s,\y) \geq \beta 1_{(-h-M/m,h+M/m)^n}(\y).
  \]
  Since $J_n$ is deterministic, we have
  \begin{align}
    \mathbb{P} \Big[ & M_n^g(s,\y) < \frac{\beta}{2} 1_{(-h-M/m,h+M/m)^n}(\y) \text{ for some } s\in [t/2m,t/m] \text{ and } \y\in \mathbb{R}^n \Big] \notag \\
    &\leq \mathbb{P} \Big[ I_n(s,\y) < -\frac{\beta}{2} 1_{(-h-M/m,h+M/m)^n}(\y) \text{ for some } s\in [t/2m,t/m] \text{ and } \y\in \mathbb{R}^n \Big] \notag \\
    &\leq \mathbb{P} \left[ \sup_{\substack{s\in [t/2m,t/m]\\ \y\in (-h-M/m,h+M/m)^n}} | I_n(s,\y) | > \frac{\beta}{2} \right] \notag \\
    &\leq \left( \frac{\beta}{2} \right)^{-p} \mathbb{E} \left[  \sup_{\substack{s\in [t/2m,t/m]\\ \y\in (-h-M/m,h+M/m)^n}} | I_n(s,\y) |^p \right] \notag \\
    &\leq \left( \frac{\beta}{2} \right)^{-p} \mathbb{E} \left[  \sup_{(s,\y) \in [t/2m,t/m]\times[-3M,3M]^n} | I_n(s,\y) |^p \right],
    \label{eq:supPthMoment}
  \end{align}
  for all $p\geq 2$ by Markov's inequality.
  We shall bound the final expectation.
  Fix $\theta \in \big(0,\frac{1}{4} - \frac{n+1}{p} \big)$ then since $I_n(0,\y) \equiv 0$ for all $\y$, we have
  \begin{align}
    \mathbb{E}\left[ \sup_{\substack{s\in [t/2m,t/m]\\ \y\in [-3M,3M]^n}} \left| \frac{I_n(s,\y)}{(t/m)^{\theta}} \right|^p \right]
    &\leq \mathbb{E} \left[ \sup_{\substack{s\in [t/2m,t/m] \\ \y\in [-3M,3M]^n}} \left| \frac{ I_n(s,\y) - I_n(0,\y) }{s^{\theta}} \right|^p \right] \notag \\
    &\leq \mathbb{E} \left[ \sup_{\substack{s,s^\prime\in [0,t/m], s\neq s^\p \\ \y \in [-3M,3M]^n}} \left| \frac{I_n(s,\y) - I_n(s^\prime,\y)}{ |s-s^\prime|^{\theta} } \right|^p \right].
      \label{eq:KolmogorovBound}
  \end{align}
  Recall that Kolmogorov's continuity criterion (see \cite[Chapter I, Theorem 2.1]{RY98}) states that for a stochastic process $(X(\mb{t}):\mb{t}\in[0,T]^d)$, if there exist strictly positive constants $C$, $\alpha$ and $p$ with $\alpha p>d$ such that
  \[
    \Vert X(\mb{s}) - X(\mb{t}) \V_p \leq C |\mb{s}-\mb{t}|^\alpha, \quad\text{for all } \mb{s},\mb{t}\in[0,T]^d,
  \]
  then $X$ has a H\"older continuous modification which satisfies for all $\theta\in[0,\alpha-d/p)$,
  \begin{equation}
    \left\V \sup_{\substack{\mb{s}\neq \mb{t} \\ \mb{s},\mb{t}\in[0,T]^d}} \frac{|X(\mb{s})-X(\mb{t})|}{|\mb{s}-\mb{t}|^\theta} \right\V_p \leq C T^{\alpha-\theta} \frac{2^{\theta+1} 2^{d/p}}{1 - 2^{d/p} 2^{-(\alpha-\theta)}}.
    \label{eq:modulusOfCty}
  \end{equation}
  Note that for $\theta$ fixed, the right hand side of (\ref{eq:modulusOfCty}) is bounded for all $p\geq 2$.
 
  From the proof of Proposition \ref{prop:MnItermCty} we see that for all $p\geq 2$ there is a constant $C := C(n)$ such that for all $(s,\y)$, $(s^\p,\y^\p) \in[0,t/m]\times[-3M,3M]^n$,
  \begin{equation}
    \V I_n(s,\y) - I_n(s^\p,\y^\p) \V_p \leq Cc_p \sup_{\substack{s\in[0,t/m] \\ \y\in[-3M,3M]^n}} \V M_n^g(s,\y)\V_p \big(|s-s^\prime|^{1/4} + |\y-\y^\p|^{1/2}\big).
    \label{eq:KolmogorovCheck}
  \end{equation}
  Then by Kolmogorov's continuity criterion, for $p > 4(n+1)$ there is a constant $K^\p := K^\p(M,m,n,t)$ such that (\ref{eq:KolmogorovBound}) is bounded by 
  \[
    (K^\p)^p c_p^p \sup_{\substack{s\in[0,t/m] \\ \y\in[-3M,3M]^n}} \V M_n^g(s,\y)\V_p^p \leq (4K^\p\sqrt{p})^p e^{Ap^3 t/m},
  \]
  for a constant $A$ depending only on $n$, where to obtain the inequality we have used the moment bound \eqref{eq:MnPthMoment} and the fact that $g\leq 1$ and $c_p \leq 2\sqrt{p}$.
  Furthermore, if $m > m_0\wedge t$ then $t/m\leq 1$ and thus for such $m$ we can, by the explicit bound on the right hand side (\ref{eq:modulusOfCty}), replace the constant $K^\p$ in the previous display with a constant $K := K(M,n)$. 
  Consequently, for all $p > 4(n+1)$ 
  \begin{align*}
    \left(\frac{\beta}{2}\right)^{-p} \mathbb{E}\left[ \sup_{\substack{s\in[t/2m,t/m] \\ \y\in[-3M,3M]}} |I_n(s,\y)|^p \right] 
    &\leq \left(\frac{8K\sqrt{p}}{\beta} \left(\frac{t}{m}\right)^\theta \right)^p e^{A p^3t/m} \\
    &\leq \exp\left( \frac{A p^3 t}{m} + p\log(8K\beta^{-1}t^\theta \sqrt{p}) - p\theta\log(m) \right).
  \end{align*}
Choose $p = 8(n+1)$ and $\theta\in(\frac{1}{p},\frac{1}{8})$ and for such choice denote the exponential in the last line above by $\delta(m)$, then for $m$ large, $\delta(m) \approx \exp(-\log(m^{\rho(n+1)}))$ with $\rho = 8\theta >1/(n+1)$ and therefore
  \[
    (1-\delta(m))^m \approx \bigg( 1 - \frac{1}{m^{\rho(n+1)}} \bigg)^m \to 1, \quad\text{as } m\to\infty,
  \]
  for all $n\geq 1$ as required.
\end{proof}

We are now ready to prove the main result of this section.

\begin{proof}[Proof of Theorem \ref{thm:strongComparision}]
  By linearity, $M_n^{1} - M_n^{2}$ is the solution to (\ref{eq:MnBounded}) with initial data $g_1-g_2$ and so it suffices to prove that $\mathbb{P}[ M_n^g(t,\y) > 0 \text{ for all } t>0 \text{ and }\y\in\mathbb{R}^n ] = 1$, for $g$ such that $g\geq 0$ and $g(\y) > 0$ for some $\y\in\mathbb{R}^n$ almost surely.
  
  Since $g$ is continuous by assumption one can find constants $c>0$, $d>0$ small enough such that $g(\x) \geq c\prod_{i=1}^n 1_{(y_i-d,y_i+d)}(\x)$ for all $\x\in\R^n$.
  Without loss of generality, we can assume $c=1$ and take $\y$ to be the origin for convenience.
  By the weak comparision principle (Lemma \ref{lem:weakComparison}), we can assume that the initial data is $g(\cdot) = 1_{(-d,d)^n}(\cdot)$. 

  Let $\gamma = \beta/2$ where $\beta$ is the constant in Lemma \ref{lem:positiveLemma1}. 
  Fix $t>0$ and $M>0$ such that $(-d,d) \subset (-M,M)$.
  For $k=1,\ldots,m$, define the events
  \begin{align*}
    & A_k := \left\{ M_n^g(s,\y) \geq \gamma^{k} 1_{(-d-\frac{Mk}{m}, d+\frac{Mk}{m})^n}(\y) \text{ for all } s\in \left[ \frac{(2k-1)t}{2m}, \frac{kt}{m} \right] \text{ and } \y\in\mathbb{R}^n \right\},
  \end{align*}
  and for $k=2,\ldots,m$ the events
  \begin{align*}
    B_1 &:= \left\{ M_n^g(t/2m,\y) \geq \gamma 1_{(-d-\frac{M}{m},d+\frac{M}{m})^n}(\y) \text{ for all } \y\in\mathbb{R}^n \right\} \\
    B_k &:= \\
    \;\; &\left\{ M_n^g(s,\y) \geq \gamma^{k} 1_{(-d-\frac{Mk}{m}, d+\frac{Mk}{m})^n}(\y) \text{ for all } s\in \left[ \frac{(k-1)t}{m}, \frac{(2k-1)t}{2m} \right] \text{ and } \y\in\mathbb{R}^n \right\}.
  \end{align*}
  We consider first the sets $A_k$.
  By Lemma \ref{lem:positiveInductionStep}, there is an $m_0 = m_0(d,M,n,t)$ such that for all $m\geq m_0$ there is a $\delta(m)$ such that
  \[
    \mathbb{P}[A_1] \geq 1 - \delta(m).
  \]
  Now assume that $A_1 \cap\cdots\cap A_{k-1}$ occurs.
  On the event $A_{k-1}$ we have $M_n^g((k-1)t/m,\y) \geq \gamma^{k-1}1_{(-d-M(k-1)/m, d+M(k-1)/m)^n}(\y)$ for all $\y\in\mathbb{R}^n$ almost surely.
  Define a time shifted white noise by $\dot{W}^k(s,\y) = \dot{W}((k-1)t/m + s,\y)$.
  Let $M_n^k(s,\y)$ be the solution driven by the noise $\dot{W}^k$ with initial data given by $\gamma^{k-1} 1_{(-d-M(k-1)/m, d+M(k-1)/m)^n}(\y)$.
  On the event $A_{k-1}$, by the weak comparison principle, $M_n^g( (k-1)t/m + s,\y) \geq M_n^k(s,\y)$ for all $s\geq 0$ and $\y\in\mathbb{R}^n$ almost surely.
It is easy to see that $\tilde{M}_n^k(s,\y) := \gamma^{-(k-1)}M_n^k(s,\y)$ is the solution to (\ref{eq:MnBounded}) with initial data $1_{(-d-M(k-1)/m, d+M(k-1)/m)^n}(\y)$.
  Lemma \ref{lem:positiveInductionStep} applied to $\tilde{M}_n^k$ with $h_0 = d$ and $h = d + M(k-1)/m$ shows that with the same $m_0$ and $\delta(\cdot)$ as above that for all $m\geq m_0$
  \[
    \mathbb{P}\left[ \tilde{M}_n^k(s,\y) \geq \gamma 1_{(-d-\frac{Mk}{m}, d+\frac{Mk}{m})^n}(\y) \text{ for all } s\in\left[ \frac{t}{2m}, \frac{t}{m} \right] \text{ and } \y\in\mathbb{R}^n \right] \geq 1 - \delta(m),
  \]
  and hence
  \[
    \mathbb{P}\left[ M_n^k(s,\y) \geq \gamma^{k} 1_{(-d-\frac{Mk}{m}, d+\frac{Mk}{m})^n}(\y) \text{ for all } s\in\left[ \frac{t}{2m}, \frac{t}{m} \right] \text{ and } \y\in\mathbb{R}^n \right] \geq 1 - \delta(m).
  \]
  By the above discussion, this implies that
  \[
    \mathbb{P}[A_k | A_1 \cap\cdots\cap A_{k-1}] \geq 1 - \delta(m) \quad\text{for } 1\leq k\leq m.
  \]
  Now since $A_1 \subseteq B_1$, $\mathbb{P}[B_1] \geq 1-\delta(m)$ and then proceeding in the same manner as before, we have
  \[
    \mathbb{P}[B_k | B_1 \cap\cdots\cap B_{k-1}] \geq 1 - \delta(m) \quad\text{for } 1\leq k\leq m.
  \]
  Finally, by the union bound
  \begin{align*}
    \mathbb{P}\left[ \bigcap_{k=1}^{m} A_k \cap \bigcap_{k=1}^{m} B_k \right] 
      &= 1 - \mathbb{P}\left[ \left( \bigcap_{k=1}^{m} A_k \right)^c \cup \left( \bigcap_{k=1}^{m} B_k \right)^c \right] \\
      &\geq 1 - \left( 1 - \mathbb{P}\left[ \bigcap_{k=1}^{m} A_k \right] \right) - \left( 1 - \mathbb{P}\left[ \bigcap_{k=1}^{m} B_k \right] \right) \\
      &\geq 2\big( 1 - \delta(m) \big)^m - 1.
  \end{align*}
  Since $(1-\delta(m))^m \to 1$ as $m\to\infty$, we conclude that
  \begin{align*}
    \mathbb{P}\left[ M_n^g(s,\y) > 0 \text{ for all } s\in\left[ \frac{t}{2}, t \right] \text{ and } \y\in \left[-\frac{M}{2}, \frac{M}{2} \right]^n \right] 
      &\geq \lim_{m\to\infty} \mathbb{P}\left[ \bigcap_{k=1}^{m} A_k \cap \bigcap_{k=1}^{m} B_k \right] \\
      &= 1.
  \end{align*}
  Since $t>0$ and $M>0$ are arbitrary, this completes the proof in the case when the initial data $g$ is a continuous function.

  We now prove part (b) of the theorem; the everywhere strict positivity of a solution $M_n(t,\x,\y)$ to \eqref{eq:MnDelta}.
  We first prove that for all $\x\in\R^n$
  \begin{equation}
    \mathbb{P}[ M_n(t,\x,\y) > 0 \text{ for all } t>0 \text{ and } \y\in\mathbb{R}^n] = 1.
      \label{eq:yPositivity}
  \end{equation}
  Let $\dot{W}^\tau(s,\y) = \dot{W}(\tau + s,\y)$ be the time shifted white noise and let $M_n^\tau$ be the solution to \eqref{eq:MnBounded} driven by the noise $\dot{W}^\tau$ and with initial data $M_n(\tau,\x,\cdot)$.
  Recall that, see \eqref{eq:MnDeltaPositive}, $\mathbb{P}[M_n(t,\x,\y) \geq 0 \text{ for all } t > 0 \text{ and } \y\in\R^n] = 1$.
  If $\mathbb{P}[M_n(\tau,x,\y)>0 \text{ for some } \y] = 1$ then since $\y\mapsto M_n(\tau,\x,\y)$ is continuous by Theorem \ref{thm:MnMain} (b), the strong comparison principle for continuous initial data proved above applied to the solution $M_n^\tau$ shows that $\mathbb{P}[M_n^\tau(s,\x,\y) > 0 \text{ for all } s>0 \text{ all } \y\in\R^n] = 1$.
  On the other hand, if $\mathbb{P}[M_n(\tau,\x,\y) = 0 \text{ for all } \y] > 0$ then $M_n^\tau(s,\x,\y) = 0$ for all $s>0$ and $\y\in\R^n$ with strictly positive probability. 
  Since $\tau$ is arbitrary, this shows that for each $\x$, either the event $\{M_n(t,\x,\y) > 0 \text{ for all } t>0 \text{ and } \y\in\R^n\}$ or the event $\{M_n(t,\x,\y) = 0 \text{ for all } t>0 \text{ and } \y\in\R^n\}$ occurs. 
  Let $p(\x)$ denote the probability of the latter event.
  By \cite[Corollary 6.2]{OW11}, we have for all $t>0$ and $\x\in\R^n$ with probability one,
  \[
    M_n(2t,\x,\x) = \frac{1}{n!} \int_{\R^n} M_n(t,\x,\y) M_n^t(t,\y,\x) \Delta(\y)^2 \;\rd{\y},
  \]
  where $M_n(t,\x,\y)$ and $M_n^t(t,\y,\x)$ are independent and by Proposition \ref{prop:MnSymmetry} equal in distribution.
  If either of the two is identically zero then so is $M_n(2t,\x,\x)$ and since the event $\{M_n(2t,\x,\x) = 0\}$ and $\{M_n(t,\x,\y) = 0 \text{ for all } (t,\y)\}$ are equal up to a null set, this implies that $p(\x) \geq 2p(\x) - p(\x)^2$.
  Hence, $p(\x)$ is either 0 or 1 but $\mathbb{E}[M(t,\x,\y)] = J_n(t,\x,\y)$ which is non-zero and so $p(\x)$ must be equal to 0. 

  Fix $\tau>0$ then \eqref{eq:yPositivity} together with Proposition \ref{prop:MnSymmetry} shows that there is a set $\mathcal{N}$ of probability zero such that on its complement, $M_n$ is jointly continuous and $M_n(\tau,\x,0)>0$ for all $\x$.
  Define $c(\x) := M_n(\tau,\x,0)/2$ and $d(\x) = \inf\{|\y| : \y\in\R^n \text{ with } M_n(\tau,\x,\y) = c(\x)\}$, then on the complement of $\mathcal{N}$, $c(\x)$ and $d(\x)$ are strictly positive and $M_n(\tau,\x,\y) \geq c 1_{(-d,d)^n}(\y)$ for all $\x$, $\y\in\R^n$.
  For $N\geq 1$, define the random set $B_N := \{\x\in\R^n : c(\x) \geq 1/N, d(\x) \geq 1/N \}$ then $M_n(\tau,\x,\y) \geq (1/N) 1_{(-1/N, 1/N)^n}(\y)$ for all $\y$ and all $\x\in B_N$.
  The strict positivity result proved above applied to the solution with initial data $(1/N) 1_{(-1/N, 1/N)^n}(\y)$ together with the weak comparison principle implies that
  \[
    \mathbb{P}[ E_N ] := \mathbb{P}[ M_n(\tau + s, \x, \y) > 0 \text{ for all } s>0 \text{ and } \y\in\R^n, \x\in B_N ] = 1.
  \]
  On the set $\mathcal{N}^c$ we have $\bigcup_{N=1}^\infty B_N  = \R^n$ otherwise there exists an $\x\in\R^n$ such that either $c(\x)=0$ or $d(\x)=0$ which is a contradiction and therefore $\mathbb{P}[ \bigcap_{N=1}^\infty E_N ] = \mathbb{P}[ M_n(\tau + s, \x, \y) > 0 \text{ for all } s>0 \text{ and } \x,\y\in\R^n ] = 1$ as required.
\end{proof}

\def\MR#1{\href{http://www.ams.org/mathscinet-getitem?mr=#1}{MR#1}}

\providecommand{\bysame}{\leavevmode\hbox to3em{\hrulefill}\thinspace}

 
\end{document}